\documentclass[10pt]{article}
\usepackage{amsmath,amssymb,hyperref,enumitem,amsfonts,amsthm,mathrsfs}
\usepackage{prooftree}
\usepackage[margin=3.5cm]{geometry}
\DeclareFontFamily{U}{mathb}{\hyphenchar\font45}
\DeclareFontShape{U}{mathb}{m}{n}{
      <5> <6> <7> <8> <9> <10> gen * mathb
      <10.95> mathb10 <12> <14.4> <17.28> <20.74> <24.88> mathb10
      }{}
\DeclareSymbolFont{mathb}{U}{mathb}{m}{n}
\DeclareFontSubstitution{U}{mathb}{m}{n}
\DeclareMathSymbol{\sqsupsetneq}  {0}{mathb}{"89}
\DeclareMathSymbol{\sqsubseteq} {0}{mathb}{"84}

\usepackage{bm}
\usepackage{mathtools}
\usepackage{bbm}
\usepackage{enumitem}
\usepackage{amsthm}
\usepackage{url}
\usepackage{tikz}
\usepackage{tikz-cd}
\usepackage{setspace}
\usepackage{cleveref}
\allowdisplaybreaks
\theoremstyle{plain}
\newtheorem{thm}{Theorem}[section]
\newtheorem{lemma}[thm]{Lemma}
\newtheorem{prop}[thm]{Proposition}
\newtheorem{cor}[thm]{Corollary}
\theoremstyle{definition}
\newtheorem{defin}[thm]{Definition}
\newtheorem{obs}[thm]{Observation}
\newtheorem{notation}[thm]{Notation}

\newtheorem{question}[thm]{Question}
\usepackage{hyperref}
\hypersetup{
    colorlinks=true,
    linkcolor=blue,
    citecolor=blue,
    filecolor=blue, 
    urlcolor=blue,       
    }

\def\vp{\varphi}

\def\atr{\mathbf{ATR}_0}
\def\atrset{\mathbf{ATR}^{\mathrm{set}}_0}

\newcommand{\B}{\ensuremath{\mathbf{B}}}
\newcommand{\C}{\ensuremath{\mathbf{C}}}

\newcommand\kp{\mathbf{KP}}

\newcommand{\zfc}{\ensuremath{\mathbf{ZFC}}}

\newcommand{\cin}{\overset{\circ}{\in}}

\DeclareMathOperator{\po}{\mathcal{P}}
\newcommand{\pair}[1]{\langle #1\rangle}

\newcommand{\Imp}{\Rightarrow}

\newcommand{\prf}[1]{\prooftree #1 \endprooftree}

\def\line{\justifies }

\title{Axiom Beta Implies Elementary Transfinite Recursion}
\author{Emanuele Frittaion \\
Department of Mathematics, University of Udine
\and 
Giorgio G. Genovesi\\
School of Mathematics, University of Leeds
}

\begin{document}
\maketitle
\tableofcontents

\begin{abstract}
We show that, over a weak base theory of sets $\B$, Axiom Beta implies Elementary or $\Delta_0$ Transfinite Recursion. The proof is novel as, unlike previous approaches, it does not make use of the Axiom of Countability or the Powerset Axiom and only requires predicative or $\Delta_0$ Separation. Furthermore, it is shown that Axiom Beta implies the existence of the relativized constructible hierarchy, the totality of the Veblen function and that of the primitive recursive set functions. In particular, it follows that the system $\B+\text{Axiom Beta}$ is equivalent to Simpson's system $\atrset$  without the Axiom of Countability \cite{simpson82} or, equivalently,  $\mathbf{PRS}\omega+\text{Axiom Beta}$. We also establish an upper bound for the $\Sigma_1$-definable functions of $\mathbf{B}+\text{Axiom Beta}$. Finally, by taking an appropriate submodel of the first Cohen model, we show that the Finite Powerset axiom in $\B+\text{Axiom Beta}$ is in a sense necessary and cannot be replaced by the closure under the rudimentary functions.
\end{abstract}

\section{Introduction} 
Axiom Beta is the statement that every well founded relation admits a collapsing function. The statement restricted to extensional relations is known as the Mostowski collapse lemma. Axiom Beta is a weak form of Replacement that naturally appears when interpreting sets as well founded trees or relations. Some examples of interpretations based on trees include  Kanovei's  intermediate  interpretation  of $\mathbf{ZFC}^-$ in $\mathbf{Z}^-$ \cite{Kanovei}, Mathias' interpretation\footnote{In Mathias' work Axiom Beta appears as a consequence of a principle he terms Axiom H. Mathias shows that Axiom H implies the Mostowski collapse lemma for extensional relations \cite[Lemma 3.0]{mathias97}, but the proof seems to work for non extensional relations as well.}  of  $\Delta_0$ Collection in Mac Lane set theory \cite{mathias97}, and Simpson's interpretation of the set theory  $\atrset$ in  the theory of second order arithmetic $\atr$ \cite{Simpson}. \smallskip 

Axiom Beta does not in general imply the scheme of Collection or Replacement. For example, the usual interpretation of $\mathbf{ZFC}^-$, that is $\mathbf{ZFC}$ without the Powerset Axiom, in second order arithmetic requires the construction of $L$, the constructible universe. In the above mentioned cases, one  shows that a theory where Axiom Beta holds is  strong enough to prove the existence of $L_\alpha$ for every ordinal $\alpha$. However, the arguments rely on additional principles.  Kanovei's construction uses  Separation, or at least $\Sigma_1$ Separation. Mathias' proof relies on the Powerset Axiom (see \cite[Remark 4.1]{mathias97}). Simpson's construction of $L$ in $\atrset$ \cite[Section VII.4]{Simpson} crucially relies on the Axiom of Countability: the proof consists in defining the usual ramified set hierarchy as a well founded relation on $\omega$ by Arithmetical Transfinite Recursion, the signature principle of the system $\atr$ of second order arithmetic, and then taking the collapse. It is therefore natural to ask whether Axiom Beta alone suffices to recover sufficient transfinite recursion for the construction of $L$  without assuming Countability, Power Set, or strong Separation principles. \smallskip

The main goal of this paper is to show that it does. More precisely, we prove that Axiom Beta implies the natural analogue of Arithmetical Transfinite Recursion, namely, the scheme of Elementary or $\Delta_0$ Transfinite Recursion, even in the absence of Countability, Power Set, and Separation beyond that for $\Delta_0$ formulas. In fact, the proof we lay out can be carried out in each of the previously mentioned systems with Axiom Beta. The usual way of showing that in $\atrset$ the reals satisfy Arithmetical Transfinite Recursion is by proving  comparability of well orderings, which in turn implies Arithmetical Transfinite Recursion.  This proof, however, heavily relies on the Kleene normal form for $\Pi^1_1$ predicates, which in uncountable settings is known not to necessarily hold. Furthermore, the implication from comparability of well orderings to Arithmetical Transfinite Recursion goes through $\Pi^1_1$ Reduction, which over theories like von Neumann-Bernays-G\"{o}del set theory ($\mathbf{NBG}$) is strictly stronger than  Elementary  Transfinite Recursion (see \cite[Corollary 12]{Sato2}). \smallskip

For our purposes, we introduce a weak set theory $\mathbf{B}$, which will serve as our base theory. The theory $\mathbf{B}$ is a subtheory of primitive recursive set theory with Infinity $(\mathbf{PRS}\omega)$. Moreover, $\mathbf{B}$ is mutually interpretable with the theory $\mathbf{ACA}_0$ of second order arithmetic.  Both $\mathbf{ACA}_0$ and $\mathbf{NBG}$ interpret end extensions satisfying $\mathbf{B}$ by a variation of the tree translation. In particular, $\mathbf{B}$ has the same second order arithmetic consequences as $\mathbf{ACA}_0$. Thus a wide range of results in $\mathbf{B}$ carry over to both $\mathbf{ACA}_0$ and $\mathbf{NBG}$. \smallskip

We then consider  the theory $$\C = \B + \text{Axiom Beta}.$$  The first main result is that $\C$, and hence, also $\mathbf{PRS}\omega+\text{Axiom Beta}$, proves Elementary Transfinite Recursion.  Therefore, the theory $\C$ is the natural formulation of $\atrset$ without the Axiom of Countability. In fact, we  show over $\B$ that Elementary Transfinite recursion is equivalent to the existence of bisimulations on well founded trees. Using the fact that bisimulations can be defined as winning strategies of clopen games, we show over $\B$  the equivalence between Clopen Determinacy and Elementary Transfinite Recursion. This is an extension of a result of Gitman and Hamkins \cite{hatman} from  $\mathbf{NBG}$ to our weaker setting.\smallskip

The second  main result is that the theory $\C$, despite the absence of Countability,  can indeed prove the existence of $L_\alpha$  for every ordinal  $\alpha$. The proof of this somewhat surprising fact provides an alternative approach to those mentioned above, and in particular to that of Simpson \cite{Simpson}, where Countability plays a pivotal role. \smallskip

On this basis,  we  next show that $\C$ proves  the totality of all primitive recursive  set functions. To this end, we first show in $\C$ the totality of the Veblen function, adapting an argument due to Marcone and Montalbán \cite{Martalban}. We then formalize within $\C$ an explicit version of the stability theorem of Jensen and Karp \cite{Jarp}, from which we obtain, provably in $\C$, the totality of all primitive recursive  set functions.     We also partially describe the $\Sigma_1$-definable set-theoretic functions of $\C$ by characterizing those of a slightly stronger theory, which is essentially $\C+ \Delta_0 \text{ Collection} + \Sigma_1 \text{ Foundation}$. For this, we will adapt arguments from Rathjen \cite{R92}. \smallskip

Finally, we show that the theory  $\C$ is synonymous with the system  $\mathbf{PRS}\omega+\text{Axiom Beta}$, thus providing a different axiomatization to  $\atrset$ without Countability which avoids any mention of primitive recursive set functions or the constructible universe. We conclude the paper by showing  that the Finite Powerset Axiom in $\C$ is, in a sense, necessary and cannot be replaced by the totality of the rudimentary functions even in the presence of the full Separation Scheme.

\section{Preliminaries}
\begin{notation}
    We will use both capital and lowercase letters to denote sets. We use angle brackets $\pair{x,y}$ to mean the Kuratowski pair $\{\{x\},\{x,y\}\}$. Tuples are defined recursively by $\pair{x_0,x_1,\dots,x_n}=\pair{x_0,\pair{x_1,\dots, x_n}}$. We will use round brackets $(x_0,\dots,x_{n-1})$ for finite sequences. To ease notation, we will write $f(x_0,\dots, x_n)$ instead of $f((x_0,\dots, x_n))$ when $f$ is a function on sequences and $f(x,y)$ instead of $f(\pair{x,y})$ when $f$ is a function on pairs. For a set $X$ we denote the collection of finite sequences of $X$ by $X^{<\omega}$ or $(X)^{<\omega}$. For any pair of finite sequences $\sigma$ and $\tau$, we will denote their concatenation by $\sigma^\frown \tau$.  We write $\sigma\sqsubseteq \tau$ to mean  there exists $\rho$ such that $\sigma^\frown \rho=\tau$.  Given a set of sequences $S$ we denote by $\sigma^\frown S$ the collection $\{\sigma^\frown \tau:\tau \in S\}\cup\{\rho: \rho\sqsubseteq \sigma\}$. Given a function $f$ we denote by $f\vert_E=\{\pair{x,y}:f(x)=y\wedge x\in E\}$ the restriction of $f$ to the set $E$. By ordinal we mean a transitive set of transitive sets; and $Ord$ will denote the class of all ordinals.  When $R$ is a relation, we will use the infix notation $\sigma\, R\, \tau$ instead of $\pair{\sigma,\tau}\in R$.
\end{notation}

\begin{defin}
    The base system we make use of is the system $\B$, which has as axioms:
    \begin{enumerate}
        \item Extensionality;
        \item Pair;
        \item Union;
        \item Infinity (there exists a limit ordinal);
        \item $\Delta_0$ Separation;
        \item Regularity;
        \item Transitive Closure;
        \item Finite Powerset (for every set $x$ there is a set $\mathcal{P}_\text{fin}(x)=\{y\subseteq x:|y|<\omega\}$);
    \end{enumerate}
    The choice of the Finite Powerset Axiom is heavily motivated by Mathias' work on weak predicative set theories, where it is shown that such axiom resolves many of the pathologies of such systems \cite[Section 8]{mathias06}. In fact, replacing the Finite Powerset Axiom in the system $\B$ with the existence of Cartesian products would yield a theory that is unable to prove the existence of the set of even numbers \cite[Theorem 2.3.3]{gandy74}. The theory $\B$ has the following properties:
    \begin{itemize}
        \item $\B$ proves for any pair of sets $x$ and $y$ the Cartesian product $x\times y$ and the set of finite sequences $x^{<\omega}$ exist as sets.
        \item $\B$ proves that the ordinals are well ordered under membership.
        \item $\B$ is finitely axiomatizable, by a single $\Pi_2$ formula, by the usual trick of replacing the axiom scheme of $\Delta_0$ Separation by the closure of the G\"{o}del operations (for a detailed exposition, see \cite[Theorem 13.4]{jech}).
        \item $\B$ is mutually interpretable with the system  $\mathbf{ACA}_0$ and has the same second order arithmetic consequences. In fact, any model of $\mathbf{ACA}_0$ can be extended to a model of $\B$. The proof makes use of a modification of the tree interpretation, which goes beyond the scope of this work.
    \end{itemize}
    Let $\C= \B+\text{ Axiom Beta}$, where Axiom Beta is the assertion that for every well founded relation $R$ on a set $X$ there exists a transitive set $Y$ and a function $\pi:X\rightarrow Y$ such that $$\forall u\in X\;(\pi(u)=\{\pi(v):v\in X\wedge \pair{v,u}\in R\}).$$
    The function $\pi$ is called the collapsing function of $R$ on $X$. We will omit the reference to either $R$ or $X$ if it is clear from the context. The system $\B$ proves that collapsing functions for well founded relations, when they exist, are unique.
    We do not require that the relation $R$ be extensional as in some formulation of the Mostowski collapse lemma. 
    \end{defin}
    \begin{defin}
        By a  tree we mean a set $T\subseteq X^{<\omega}$, for some set $X$, that is closed under initial segments. Given two trees $T_0$ and $T_1$, we write 
        $$Pair(T_0,T_1)=(0)^\frown T_0\cup(1)^\frown T_1,$$
        and 
        $$OPair(T_0,T_1)=Pair(Pair(T_0,T_0),Pair(T_0,T_1)).$$
        When considering the collapse of a tree $T$, we consider the relation $R$ on the tree given by immediate successor rather than initial segment. Given a tree $T$ and its collapsing function $\pi$, we will write $\pi T$ to mean $\pi(\emptyset)$ and, to avoid excessive parentheses, we will write $\pi \sigma$ to mean $\pi(\sigma)$. Since the collapsing function for a well founded tree is unique, we will not use distinct notation for the collapsing function of distinct trees. Using this convention, if $T_0$ and $T_1$ are well founded, then
        $$\pi Pair(T_0,T_1)=\{\pi T_0,\pi T_1\}\quad\text{ and }\quad\pi OPair(T_0,T_1)=\{\{\pi T_0\},\{\pi T_0,\pi T_1\}\}=\pair{\pi T_0,\pi T_1}.$$
        
    \end{defin}
    \begin{defin}\label{bisimdef}
    Given a tree $T\subseteq X^{<\omega}$, a bisimulation on $T$ is a relation $B\subseteq T\times T$ such that
    $$ \sigma\, B \,\tau \leftrightarrow (\forall\rho_0\sqsupsetneq \sigma\,\exists \rho_1\sqsupsetneq \tau\,(\rho_0\,B\,\rho_1))\wedge (\forall\rho_1\sqsupsetneq \tau\,\exists \rho_0\sqsupsetneq \sigma\,(\rho_0\,B\,\rho_1)).$$
    \end{defin}
    Every well founded tree has a unique bisimulation which will also satisfy the converse of the above implication\footnote{This is true of any maximal bisimulation.}. In general, however, there may be many bisimulations. If $\pi:T\rightarrow Y$ is the collapsing function of a well founded tree $T$, then the set
    $$\{\pair{\sigma,\tau}:\pi(\sigma)=\pi(\tau)\}$$
    is a bisimulation by the Axiom of Extensionality. 

\noindent
A natural principle to consider is the following
$$\text{Bisimulation Axiom:} \quad  \text{ Every well founded tree has a bisimulation}.$$
By the previous remark and $\Delta_0$ Separation, we have $\C\vdash \text{Bisimulation Axiom}$. In the context of reverse mathematics, the Bisimulation Axiom implies over $\mathbf{RCA}_0$ the comparability of well orders, so $\mathbf{RCA}_0+\text{ Bisimulation Axiom}$  is equivalent to  $\atr$.
We will also consider the Axiom of Countability or $V=H(\aleph_1)$, which is the statement that every set has countable transitive closure.
\begin{defin}
    The tree translation, following Simpson's definition \cite[Definition VII.3.13]{Simpson}, is that which has as domain the well founded trees, the translation of equality $T=^* S$ is the formula: $$\pair{(0),(1)}\in B\text{ where }B \text{ is a bisimulation on }Pair(T,S)$$ and the translation of membership $T\in ^* S$ is the formula:
    $$\exists \sigma \in S\;(|\sigma|=1\wedge \{\tau:\sigma^\frown \tau \in S\}=^* T).$$
    One can show that this translation supports an interpretation of $\C$ in $\B+\text{Bisimulation Axiom}$. 
\end{defin}

\section{Elementary diagrams}
We show that the theory $\C$ proves the existence of $Th(a,\in)$, namely the elementary diagram of the structure $(a,\in)$. By a similar argument, one can  show that the theory $\C$ also proves the existence of the elementary diagram $Th(M)$ of every set-sized structure $M$. For our purposes and for the constructions that follow, however,  it is sufficient to consider the paradigmatic case where $M$ is just $(a,\in)$. The main idea, which appears to be new, is to define \emph{truth values} $\top_\theta$ and $\bot_\theta$ 
that depend on the formula itself. These values turn out to be well founded trees 
which do not conflate truth; that is, $\pi\top_\theta \neq \pi\bot_\theta$. Moreover,
\[
\pair{\theta,\sigma}\in Th(a,\in)
\quad \text{if and only if} \quad 
\pi S_{\pair{\theta,\sigma}} = \pi\top_\theta,
\]
where $S_{\pair{\theta,\sigma}}$ is again a well founded tree depending also on the given 
variable assignment $\sigma$. Similarly,
\[
\pair{ \theta,\sigma }\notin Th(a,\in)\ 
\quad \text{if and only if} \quad 
\pi S_{\pair{\theta,\sigma}} = \pi\bot_\theta.
\]

We will apply this idea in the context of Elementary Transfinite Recursion and the constructible universe. A brief warning is in order. A common feature of all these proofs is the recursive construction of trees. Although our base theory $\C$ does not support full recursion, in each instance the required definitions can already be justified within $\B$. Indeed, they all conform to a general pattern, and we leave it to the reader to verify in each case that the construction fits the following scheme.

\begin{thm}[$\B$]
Let $\pair{A,\prec}$ be a well founded relation and $X$ be a set. Given a $\Delta_0$ formula  $\vartheta(x,y,z)$, we can define trees $T_s\subseteq X^{<\omega}$ for every $s\in A$ by 
\[ T_s=\bigcup\{ \sigma^\frown T_q : q\prec s  \land \vartheta(\sigma,q,s)\}.\]
Moreover, suppose that for every $s\in A$ there exists $l\in\omega$ such that $\vartheta(\sigma,q,s)$ with $q\prec s$ implies $|\sigma|< l$. Then every $T_s$ is well founded. 
\end{thm}
\begin{proof}
It is easy to see that the relation $\tau\in T_s$ is $\Delta_0$. Given  $Y\subseteq T_s$ nonempty, consider the set $B=\{q\in A : \exists \tau \in Y \, \exists \rho\, (\rho\sqsubseteq  \tau\land \tau\in \rho^\frown T_q)\}$. If $B$ is empty, then $Y\subseteq \{\sigma\colon |\sigma|<l\}$ for some $l\in\omega$, and therefore there is a minimal node.  Suppose not, and let $q$ be minimal in $B$. Fix $\rho$ such that $Y\cap \rho^\frown T_q \neq\emptyset $. As before,  let $l\in\omega$ bound the length of every stem of $T_{q}$, that is,  of every $\sigma$ such that $\vartheta(\sigma, r,q)$ for some $r\prec q$. Then $Y\cap \rho^\frown T_q\subseteq \{\tau : |\tau|< |\rho|+l\}$ and every $\tau\in Y\cap \rho^\frown T_q$  of maximal length is minimal in $Y$. 
\end{proof}

\begin{defin}
    We assume a fixed coding for formulas in the language of set theory; we will not use distinct notation for formulas and codes for formulas. Let $\{x_i:i\in \omega\}$ denote the set of variables indexed by $\omega$. By $FV(\theta)$ we mean the set of indices of free variables in the formula $\theta$. A variable assignment is a sequence $\sigma\in a^{<\omega}$. For a set $a$ we denote by $Th(a,\in)$ the unique set of pairs $\pair{\theta,\sigma}$, where $\sigma\in a^{<\omega}$ is a variable assignment such that $|\sigma|> \max FV(\theta)$, such that $Th(a,\in)$ agrees with atomic facts: 
    \begin{enumerate}
        \item  $\pair{x_i\in x_j, \sigma}\in Th(a,\in)\;\leftrightarrow\; \sigma(i)\in \sigma(j)$
        \item $\pair{x_i = x_j, \sigma}\in Th(a,\in)\;\leftrightarrow\; \sigma(i)= \sigma(j)$
    \end{enumerate}
    and satisfies the Tarski conditions:
 \begin{enumerate}
        \item  $\pair{\theta\wedge \psi, \sigma}\in Th(a,\in)\;\leftrightarrow\; (\pair{\theta, \sigma}\in Th(a,\in)\wedge \pair{\psi, \sigma}\in Th(a,\in))$
        \item  $\pair{\theta\vee \psi, \sigma}\in Th(a,\in)\;\leftrightarrow\; (\pair{\theta, \sigma}\in Th(a,\in)\vee \pair{ \psi, \sigma}\in Th(a,\in))$
        \item   $\pair{\neg \theta, \sigma}\in Th(a,\in)\;\leftrightarrow\; \pair{\theta, \sigma}\notin Th(a,\in)$
        \item  $\pair{\exists x_i\,  \theta, \sigma}\in Th(a,\in)\;\leftrightarrow \;\exists b\in a\;\pair{\theta, \sigma[b/x_i]}\in Th(a,\in)$
         \item  $\pair{\forall  x_i\,  \theta, \sigma}\in Th(a,\in)\;\leftrightarrow\; \forall b\in a\;\pair{\theta, \sigma[b/x_i]}\in Th(a,\in)$
    \end{enumerate}
    where $$\sigma[b/x_i]=\begin{cases}(\sigma\setminus\{\pair{i, \sigma(i)}\})\cup \{\pair{i,b}\}) \;\;\text{ if }\;\;i<|\sigma|\\
    \sigma\cup\bigcup_{|\sigma|\leq j\leq i}\{\pair{j,b}\} \,\quad \quad\quad\text{ if }\; \; i\geq|\sigma|
    \end{cases}$$ that is, we change the value of $\sigma$ at $i$ to $b$ if it is defined at $i$, otherwise we extend $\sigma$ so that at $i$ it has value $b$. We will sometimes write $\theta[b_0,\dots, b_n]$ to mean the pair $\pair{\theta,(b_0,\dots, b_n)}$.  We will also use other lowercase letters ($u,v,y,z,w\dots$) to denote variables and, for any variable assignment $\sigma$, we will write $\sigma(u)$ to mean $\sigma(i)$ where $i$ is the index for the variable $u$. We will also use the standard satisfaction predicate $a\models \theta[u_0,\dots,u_n]$ to mean $\pair{\theta,(u_0,\dots, u_n)}\in Th(a,\in)$. 
\end{defin}

\begin{thm}\label{thisatrue}
    $\C$ proves that for all $a$, $Th(a,\in)$ exists.
\end{thm}
\begin{proof}
We consider all formulas to be in negation normal form, that is, negation only appears in front of atomic formulas. We define for each formula $\theta$ trees $\top_\theta$, $\bot_\theta$, and for evaluation of the variables $\sigma$ a tree $S_{\pair{\theta,\sigma}}$ such that $$(\pair{\theta,\sigma}\in Th(a,\in))\;\leftrightarrow\; (\pi S_{\pair{\theta,\sigma}}=\pi\top_\theta) \;\leftrightarrow\; (\pi S_{\pair{\theta,\sigma}}\neq\pi\bot_\theta),$$
where $\pi$ denotes the collapsing function. We describe a construction by recursion on the complexity of $\theta$. Hereinafter,  $\sigma$ is a variable assignment that includes all the free variables of the formula being considered.
\begin{enumerate}
    \item For $\theta$ an atomic or the negation of an atomic formula define $\top_\theta=\{\emptyset, (0),(0,0),(1)\}$,  $\bot_\theta=\{\emptyset,(0),(0,0)\}$, and 
    $$S_{\pair{\theta,\sigma}}=\begin{cases}
        \top_{\theta}\;\;\ \text{ if } \pair{\theta,\sigma} 
 \text{ is true} \\
 \bot_{\theta}\;\;\; \text{ otherwise}
 \end{cases}$$ We note that $\pi \top_{\theta}=\{\{\emptyset\},\emptyset\}$ and $\pi \bot_{\theta}=\{\{\emptyset\}\}$. Such choice is to ensure the two have the same rank.
    
    \item $\top_{\theta_0\wedge \theta_1}$ is the union of the following:
    \begin{enumerate}
        \item $(\pair{\theta_0\wedge\theta_1,1,1})^\frown OPair(\top_{\theta_0},\top_{\theta_1})$
        \item $(\pair{\theta_0\wedge\theta_1,1,0})^\frown OPair(\top_{\theta_0},\bot_{\theta_1})$
        \item $(\pair{\theta_0\wedge\theta_1,0,1})^\frown OPair(\bot_{\theta_0},\top_{\theta_1})$
        \item $(\pair{\theta_0\wedge\theta_1,0,0})^\frown OPair(\bot_{\theta_0},\bot_{\theta_1})$
    \end{enumerate}

In pictures, 

    \[  \top_{\theta_0\land\theta_1}\quad=\quad \raisebox{-0.5\height}{	\begin{tikzpicture}[
		level distance=12mm,
		sibling distance=30mm,
		every node/.style={font=\normalsize},
		edge from parent/.style={draw}
		]
		
		\node {$\bullet$}
		child { node {$\pair{\top_{\theta_0},\top_{\theta_1}}$} }
		child { node {$\pair{\top_{\theta_0},\bot_{\theta_1}}$} }
		child { node {$\pair{\bot_{\theta_0},\top_{\theta_1}}$} }
		child { node {$\pair{\bot_{\theta_0},\bot_{\theta_1}}$} };
\end{tikzpicture} }   \] 
    
    $\bot_{\theta_0\wedge \theta_1}$ will be the union of the following:
    \begin{enumerate}
        \item $(\pair{\theta_0\wedge\theta_1,1,0})^\frown OPair(\top_{\theta_0},\bot_{\theta_1})$
        \item $(\pair{\theta_0\wedge\theta_1,0,1})^\frown OPair(\bot_{\theta_0},\top_{\theta_1})$
        \item $(\pair{\theta_0\wedge\theta_1,0,0})^\frown OPair(\bot_{\theta_0},\bot_{\theta_1})$
    \end{enumerate}

    \[  \bot_{\theta_0\land\theta_1}\quad=\quad \raisebox{-0.5\height}{	\begin{tikzpicture}[
		level distance=12mm,
		sibling distance=30mm,
		every node/.style={font=\normalsize},
		edge from parent/.style={draw}
		]
		
		\node {$\bullet$}
		child { node {$\pair{\top_{\theta_0},\bot_{\theta_1}}$} }
		child { node {$\pair{\bot_{\theta_0},\top_{\theta_1}}$} }
		child { node {$\pair{\bot_{\theta_0},\bot_{\theta_1}}$} };
\end{tikzpicture} }   \] 
     $S_{ \pair{\theta_0\wedge \theta_1,\sigma}}$ will be the union of the following 
  \begin{enumerate}
        \item $(\pair{\theta_0\wedge\theta_1,\sigma,1,1})^\frown OPair(S_{\pair{\theta_0,\sigma}},S_{\pair{\theta_1,\sigma}})$
        \item $(\pair{\theta_0\wedge\theta_1,\sigma,1,0})^\frown OPair(\top_{\theta_0},\bot_{\theta_1})$
        \item $(\pair{\theta_0\wedge\theta_1,\sigma,0,1})^\frown OPair(\bot_{\theta_0},\top_{\theta_1})$
        \item $(\pair{\theta_0\wedge\theta_1,\sigma,0,0})^\frown OPair(\bot_{\theta_0},\bot_{\theta_1})$
    \end{enumerate}

    \[  S_{\pair{\theta_0\land\theta_1,\sigma}}\quad=\quad \raisebox{-0.5\height}{	\begin{tikzpicture}[
		level distance=12mm,
		sibling distance=30mm,
		every node/.style={font=\normalsize},
		edge from parent/.style={draw}
		]
		
		\node {$\bullet$}
		child { node {$\pair{S_{\pair{\theta_0,\sigma}},S_{\pair{\theta_1,\sigma}}}$} }
		child { node {$\pair{\top_{\theta_0},\bot_{\theta_1}}$} }
		child { node {$\pair{\bot_{\theta_0},\top_{\theta_1}}$} }
		child { node {$\pair{\bot_{\theta_0},\bot_{\theta_1}}$} };
\end{tikzpicture} }   \]

Before moving on to describe the other trees, we will provide some motivation for the construction in the case where $\theta_0$ and $\theta_1$ are standard formulas. Assume that we have built trees $\top_{\theta_i}$, $\bot_{\theta_i}$, and $S_{\pair{\theta_i,\sigma}}$, for  $i\leq 1$,  such that
$$\pair{\theta_i,\sigma}\in Th(a,\in) \leftrightarrow (\pi S_{\pair{\theta_i,\sigma}}=\pi \top_{\theta_i}) \leftrightarrow (\pi S_{\pair{\theta_i,\sigma}}\neq \pi \bot_{\theta_i}). $$
By construction of $\top_{\theta_0\wedge \theta_1}$, $\bot_{\theta_0\wedge \theta_1}$, and $S_{\pair{\theta_0\wedge \theta_1,\sigma}}$ if $\theta_0\wedge \theta_1$ evaluated at $\sigma$ is true, then by hypothesis 
$$\pi OPair( S_{\pair{\theta_0,\sigma}}, S_{\pair{\theta_1,\sigma}})=\pair{\pi S_{\pair{\theta_0,\sigma}},\pi S_{\pair{\theta_1,\sigma}}}=\pair{\pi\top_{\theta_0},\pi\top_{\theta_1}}=\pi OPair(\top_{\theta_0},\top_{\theta_1})$$
which implies $\pi S_{\pair{\theta_0\wedge \theta_1,\sigma}}=\pi \top_{\theta_0\wedge \theta_1}\neq \pi \bot_{\theta_0\wedge \theta_1}$. The other implications are  proved similarly.

    \item $\top_{\theta_0\vee \theta_1}$ will be the union of the following:
    \begin{enumerate}
        \item $(\pair{\theta_0\vee\theta_1,1,1})^\frown OPair(\top_{\theta_0},\top_{\theta_1})$
        \item $(\pair{\theta_0\vee\theta_1,1,0})^\frown OPair(\top_{\theta_0},\bot_{\theta_1})$
        \item $(\pair{\theta_0\vee\theta_1,0,1})^\frown OPair(\bot_{\theta_0},\top_{\theta_1})$  
    \end{enumerate}
    \[  \top_{\theta_0\lor\theta_1} \quad=\quad \raisebox{-0.5\height}{%
        \begin{tikzpicture}[
		level distance=12mm,
		sibling distance=30mm,
		every node/.style={font=\normalsize},
		edge from parent/.style={draw}]
		\node {$\bullet$}
		child { node {$\pair{\top_{\theta_0},\top_{\theta_1}}$} }
		child { node {$\pair{\top_{\theta_0},\bot_{\theta_1}}$} }
		child { node {$\pair{\bot_{\theta_0},\top_{\theta_1}}$} };
        \end{tikzpicture}%
        }    
\] 
    $\bot_{\theta_0\vee \theta_1}$ will be the union of the following:
    \begin{enumerate}
         \item $(\pair{\theta_0\vee\theta_1,1,1})^\frown OPair(\top_{\theta_0},\top_{\theta_1})$
        \item $(\pair{\theta_0\vee \theta_1,1,0})^\frown OPair(\top_{\theta_0},\bot_{\theta_1})$
        \item $(\pair{\theta_0\vee\theta_1,0,1})^\frown OPair(\bot_{\theta_0},\top_{\theta_1})$
        \item $(\pair{\theta_0\vee\theta_1,0,0})^\frown OPair(\bot_{\theta_0},\bot_{\theta_1})$
    \end{enumerate}

    \[  \bot_{\theta_0\lor\theta_1}\quad=\quad \raisebox{-0.5\height}{	\begin{tikzpicture}[
		level distance=12mm,
		sibling distance=30mm,
		every node/.style={font=\normalsize},
		edge from parent/.style={draw}
		]
		
		\node {$\bullet$}
	child { node {$\pair{\top_{\theta_0},\top_{\theta_1}}$} }
	child { node {$\pair{\top_{\theta_0},\bot_{\theta_1}}$} }
		child { node {$\pair{\bot_{\theta_0},\top_{\theta_1}}$} }
	child { node {$\pair{\bot_{\theta_0},\bot_{\theta_1}}$} };
\end{tikzpicture} }   
\] 

    $S_{ \pair{\theta_0\vee \theta_1,\sigma}}$  will be the union of the following
    \begin{enumerate}
        \item $(\pair{\theta_0\vee\theta_1,\sigma,1,1})^\frown OPair(\top_{\theta_0},\top_{\theta_1})$
        \item $(\pair{\theta_0\vee\theta_1,\sigma,1,0})^\frown OPair(\top_{\theta_0},\bot_{\theta_1})$
        \item $(\pair{\theta_0\vee\theta_1,\sigma,0,1})^\frown OPair(\bot_{\theta_0},\top_{\theta_1})$  
        \item $(\pair{\theta_0\vee\theta_1,\sigma,0,0})^\frown OPair(S_{\pair{\theta_0,\sigma}},S_{\pair{\theta_1,\sigma}})$
    \end{enumerate}

\[  S_{\pair{\theta_0\lor\theta_1,\sigma}}\quad=\quad \raisebox{-0.5\height}{	\begin{tikzpicture}[
		level distance=12mm,
		sibling distance=30mm,
		every node/.style={font=\normalsize},
		edge from parent/.style={draw}
		]
		
		\node {$\bullet$}
		child { node {$\pair{\top_{\theta_0},\top_{\theta_1}}$} }
		child { node {$\pair{\top_{\theta_0},\bot_{\theta_1}}$} }
		child { node {$\pair{\bot_{\theta_0},\top_{\theta_1}}$} }
		child { node {$\pair{S_{\pair{\theta_0,\sigma}},S_{\pair{\theta_1,\sigma}}}$} };
\end{tikzpicture} }   
\] 
    
    \item $\top_{\forall x\, \theta}=(\forall x\, \theta,1)^\frown \top_{\theta}$\smallskip

    $\bot_{\forall x\,\theta}= (\forall x\,\theta,1)^\frown \top_{\theta}\cup (\forall x\, \theta,0)^\frown\bot_{\theta}$

\[  \top_{\forall x\, \theta}\quad=\quad \raisebox{-0.5\height}{	\begin{tikzpicture}[
		level 1/.style={level distance =6mm},
        level 2/.style={level distance =12mm},
		sibling distance=12mm,
		every node/.style={font=\normalsize},
		edge from parent/.style={draw}
		]
		
		\node {$\bullet$}
        child { node {$\bullet$}
		child { node {$\top_\theta$} }};
\end{tikzpicture} }  \qquad\quad 
\bot_{\forall x\, \theta}\quad=\quad \raisebox{-0.5\height}{	\begin{tikzpicture}[
		level 1/.style={level distance =6mm},
        level 2/.style={level distance =12mm},
		sibling distance=12mm,
		every node/.style={font=\normalsize},
		edge from parent/.style={draw}
		]
		
		\node {$\bullet$}
        child{ node {$\bullet$}
		child { node {$\top_\theta$} }
		child { node {$\bot_\theta$} }};
\end{tikzpicture} }   \] 

     $S_{ \pair{\forall x\,\theta,\sigma}}= (\pair{\forall x\,\theta,\sigma},1)^\frown \top_{\theta}\cup \bigcup_{s\in a} (\pair{\forall x\,\theta,\sigma},\pair{1,s})^\frown S_{\pair{\theta,\sigma[s/x]}})$.
    \[  S_{\pair{\forall x\, \theta,\sigma}}\quad=\quad \raisebox{-0.5\height}{	\begin{tikzpicture}[
		level 1/.style={level distance =6mm},
        level 2/.style={level distance =12mm},
		sibling distance=30mm,
		every node/.style={font=\normalsize},
		edge from parent/.style={draw}
		]
		
		\node {$\bullet$}
        child { node {$\bullet$}
		child { node {$\top_{\theta}$} }
		child { node {$S_{\pair{\theta, \sigma[d/x]}}$} }
		child { node {$\cdots$} edge from parent[dotted] 
			node[midway, right, xshift=20pt] {\text{all} $d\in a$} }};
\end{tikzpicture} }   \]  
    \item $\top_{\exists x\,\theta}= (\exists x\,\theta,1)^\frown \top_{\theta}\cup (\exists x\, \theta,0)^\frown\bot_{\theta}$ \smallskip

    $\bot_{\exists x\,\theta}=(\exists x\, \theta,0)^\frown \bot_{\theta}$

   \[  \top_{\exists x\, \theta}\quad=\quad \raisebox{-0.5\height}{	\begin{tikzpicture}[
		level 1/.style={level distance =6mm},
        level 2/.style={level distance =12mm},
		sibling distance=12mm,
		every node/.style={font=\normalsize},
		edge from parent/.style={draw}
		]
		
		\node {$\bullet$}
        child{ node {$\bullet$}
		child { node {$\top_\theta$} }
		child { node {$\bot_\theta$} }};
\end{tikzpicture} }   \qquad\quad  
\bot_{\exists x\, \theta}\quad= \quad\raisebox{-0.5\height}{	\begin{tikzpicture}[
		level 1/.style={level distance =6mm},
        level 2/.style={level distance =12mm},
		sibling distance=12mm,
		every node/.style={font=\normalsize},
		edge from parent/.style={draw}
		]
		
		\node {$\bullet$}
        child {node {$\bullet$}
		child { node {$\bot_\theta$} }};
\end{tikzpicture} }   \] 

    $S_{\pair{ \exists x\,\theta,\sigma}}= (\pair{\exists x\,\theta,\sigma},0)^\frown \bot_{\theta}\cup\bigcup_{s\in a} (\pair{\exists x\,\theta,\sigma},\pair{1,s})^\frown S_{\pair{ \theta,\sigma[s/x]}})$

    \[  S_{\pair{\exists x\, \theta,\sigma}}\quad=\quad \raisebox{-0.5\height}{	\begin{tikzpicture}[
		,
		level 1/.style={level distance =6mm},
        level 2/.style={level distance =12mm},
		sibling distance=30mm,
		every node/.style={font=\normalsize},
		edge from parent/.style={draw}
		]
		
		\node {$\bullet$}
        child { node {$\bullet$}
		child { node {$\bot_{\theta}$} }
		child { node {$S_{\pair {\theta, \sigma[d/x]}}$} }
	child { node {$\cdots$} edge from parent[dotted] 
		node[midway, right, xshift=20pt] {\text{all} $d\in a$} }};
\end{tikzpicture} }   \] 
\end{enumerate}
 
We can justify the above recursion in $\B$. In fact, there is a large enough set $X$ such that $X^{<\omega}$  contains as distinct subtrees all the trees described above. In more detail, let $Form$ denote the set of formulas.  The set of variable assignments is by definition $a^{<\omega}$. The tree
$$T= \bigcup_{\theta\in Form,\sigma \in a^{<\omega}}(\theta,\top)^\frown\top_\theta\cup(\theta,\bot)^\frown\bot_\theta\cup  (\theta,\sigma)^\frown S_{\pair{\theta,\sigma}}$$
is a $\Delta_0$-definable subtree of $X^{<\omega}$, where $$X= a\cup \omega \cup \omega^2\cup Form\times a^{<\omega}\cup Form\times \{\top,\bot\}.$$
Indeed, by the bookkeeping used to define the tree nodes,  we may test locally whether a sequence of $X^{<\omega}$ belongs to $T$.  So the theory $\B$ is sufficient to prove the existence of the tree $T$, which is well founded since $S_{\pair{\theta,\sigma}},\top_\theta$, and $\bot_\theta$ all have finite height. Let $\pi: T\rightarrow y$ be the collapsing function on $T$. The collection 
$$H=\{\pair{\theta,\sigma}:\pi(\theta,\sigma)=\pi(\theta,\top)\}$$
is $\Delta_0$ relative to $\pi$ and $T$ so it exists by $\Delta_0$ Separation. We observe that each tree $S_{\pair{\theta,\sigma}},\top_\theta$, and $\bot_\theta$ can be identified as a subtree in $T$ so we have
 $$\pi S_{\pair{\theta,\sigma}}=\pi(\theta,\sigma)\;,\;\pi \top_{\theta}=\pi(\theta,\top)\;\text{, and }\;\pi \bot_{\theta}=\pi(\theta,\bot).$$
We therefore have
$$\pi(\theta,\sigma)=\pi(\theta,\top)\;\leftrightarrow\; \pi S_{\pair{\theta,\sigma}}=\pi \top_\theta$$ 
and similarly
$$\pi(\theta,\sigma)=\pi(\theta,\bot)\;\leftrightarrow\; \pi S_{\pair{\theta,\sigma}}=\pi \bot_\theta$$ 
We would like to show that $H=Th(a,\in)$, that is, $H$ agrees with atomic facts and satisfies the Tarski conditions. This follows from an examination of the construction, we sketch the details.
We will also show, by induction on the complexity of the formula, that
$$\pi S_{\pair{\theta,\sigma}}=\pi\top_\theta\;\leftrightarrow\; \pi S_{\pair{\theta,\sigma}}\neq \pi\bot_\theta$$ which implies in particular that $\pi \top_\theta\neq \pi\bot_\theta$.  $H$ agrees on atomic facts of $(a,\in)$ by construction of $\top_{\theta},\bot_{\theta},S_{\pair{\theta,\sigma}}$ where $\theta$ is atomic or the negation of an atomic formula.
We will show the case for disjunction; the case for conjunction is very similar and was already hinted at in the construction. Assume that for each $i\leq 1$
$$ \pair{\theta_i,\sigma}\in H\; \leftrightarrow \; (\pi S_{\pair{\theta_i,\sigma}}= \pi\top_{\theta})\;\leftrightarrow\; (\pi S_{\pair{\theta_i,\sigma}}\neq \pi\bot_{\theta})
$$
We therefore have that the sets $$\pi OPair(\top_{\theta_0},\top_{\theta_1}),\pi OPair(\top_{\theta_0},\bot_{\theta_1}),\pi OPair(\bot_{\theta_0},\top_{\theta_1}),\text{ and }\pi OPair(\bot_{\theta_0},\bot_{\theta_1})$$
will be pairwise distinct under.
By construction, we have
\begin{enumerate}
    \item $\pi\top_{\theta_0\vee \theta_1}= \{\pair{\pi\top_{\theta_0},\pi\top_{\theta_1}},\pair{\pi\top_{\theta_0},\pi\bot_{\theta_1}},\pair{\pi\bot_{\theta_0},\pi\top_{\theta_1}}\}$
\item $\pi\bot_{\theta_0\vee \theta_1}= \{\pair{\pi\top_{\theta_0},\pi\top_{\theta_1}},\pair{\pi\top_{\theta_0},\pi\bot_{\theta_1}},\pair{\pi\bot_{\theta_0},\pi\top_{\theta_1}},\pair {\pi\bot_{\theta_0},\pi \bot_{\theta_1}}\}$
\item $\pi S_{\pair{\theta_0\vee \theta_1,\sigma}}= \{\pair{\pi\top_{\theta_0},\pi\top_{\theta_1}},\pair{\pi\top_{\theta_0},\pi\bot_{\theta_1}},\pair{\pi\bot_{\theta_0},\pi\top_{\theta_1}},\pair {\pi S_{\pair{\theta_0,\sigma}},\pi S_{\pair{ \theta_1,\sigma}}}\}$
\end{enumerate}
So in particular, we have that
$$\pi S_{\pair{\theta_0\vee \theta_1,\sigma}}=\pi \top_{\theta_0\vee \theta_1}\;\;\leftrightarrow\;\; \pair {\pi S_{\pair{\theta_0,\sigma}},\pi S_{\pair{\theta_1,\sigma}}}\in \{\pair{\pi\top_{\theta_0},\pi\top_{\theta_1}},\pair{\pi\top_{\theta_0},\pi\bot_{\theta_1}},\pair{\pi\bot_{\theta_0},\pi\top_{\theta_1}}\}$$
so 
$$(\pair {\pi S_{\pair{\theta_0,\sigma}},\pi S_{\pair{ \theta_1,\sigma}}}\neq \pair{\bot_{\theta_0},\bot_{\theta_1}})\;\leftrightarrow \;(\pi S_{\pair{\theta_0\vee \theta_1,\sigma}}=\pi \bot_{\theta_0\vee \theta_1})\;\leftrightarrow\; (\pi S_{\pair{\theta_0\vee \theta_1,\sigma}}=\pi \top_{\theta_0\vee \theta_1})$$
so 
$$(\pair{\theta_0\vee\theta_1,\sigma}\in H)\;\leftrightarrow\; (\pi S_{\pair{\theta_0\vee \theta_1,\sigma}}=\pi \top_{\theta_0\vee \theta_1})\;\leftrightarrow\; (\pi S_{\pair{\theta_0,\sigma}}=\pi\top_{\theta_0}\vee\pi S_{\pair{\theta_1,\sigma}}=\pi\top_{\theta_0})$$ 
which shows that $H$ satisfies the Tarski condition for disjunction.\\
\\
We show that $H$ satisfies the Tarski condition for the existential quantifier; the case for the universal quantifier is almost the same. Assume that for all variable assignments $\sigma$, which has index for $x$ in its domain
$$(\pair{\theta,\sigma}\in H)\;\;\leftrightarrow\;\; (\pi S_{\pair{\theta,\sigma}}=\pi \top_\theta)\;\;\leftrightarrow\;\;( \pi S_{\pair {\theta,\sigma}}\neq \pi \bot_\theta)$$
notice that this assumption is $\Delta_0$ relative to $T,\pi,H$ and therefore we only require $\Delta_0$ induction on $\omega$. In particular, $\top_\theta\neq \bot_\theta$.
\begin{enumerate}
    \item $\pi\top_{\exists x\theta}= \{\{\pi\top_{\theta},\pi\bot_{\theta}\}\}$
\item $\pi\bot_{\exists x\, \theta}= \{\{\pi\bot_\theta\}\}$
\item $\pi S_{\pair{\exists \theta,\sigma}}= \{\{\pi\bot_{\theta}\}\cup\{ \pi S_{\pair{\theta,\sigma[s/x]}}:s\in a\}\}$
\end{enumerate}
By the inductive hypothesis 
$$\forall s\in a\, (\pi S_{\pair{\theta,\sigma[s/x]}}\in \{\pi \top_\theta,\pi\bot_\theta\})$$
so
$$(\pi S_{\pair{\exists \theta,\sigma}}=\pi\top_{\exists \theta})\;\;\leftrightarrow\;\; (\exists s\in a\, \pi S_{\pair{\theta,\sigma[s/x]}}=\pi \top_{\theta})\;\;\leftrightarrow\;\;( \pi S_{\pair{\exists \theta,\sigma}}\neq\pi\bot_{\exists x\, \theta})$$
Therefore, $H$ satisfies the Tarski condition for the existential quantifier. This shows that $H$ agrees with atomic facts and satisfies the Tarski conditions, so $H=Th(a,\in)$.
\end{proof}

\begin{obs}\label{importantlater}
    In the above construction, for all $\theta$ and $\psi$ we have $\pi \top_\theta\neq \pi\bot_\psi$. This is proved by $\Delta_0$ induction on the complexity of the formula. By construction, any two formulas $\theta$ and $\psi$ have the same rank if and only if the trees $\top_\theta$, $\bot_\theta$, $\top_\psi$, and $\bot_\psi$ have the same height.
\end{obs}
\begin{obs}
    The above proof can be readily modified to work in $\B+\text{Bisimulation Axiom}$. Furthermore, 
    by the Axiom of  Infinity and the existence of elementary diagrams, the theory $\B+\text{Bisimulation Axiom}$ proves the consistency of $\mathbf{ACA}_0$. As mentioned before, the theory $\B$ and $\mathbf{ACA}_0$ are mutually interpretable and finitely axiomatized. This implies that $$\B+\text{Bisimulation Axiom}\vdash Con(\B)$$ and therefore $\B\not\vdash \text{Bisimulation Axiom}$. 
\end{obs}

\section{Transfinite recursion}
We will now use an iteration of the construction of the previous section to prove an analogue of arithmetical transfinite recursion. A similar statement known as  Elementary Transfinite Recursion ($\mathbf{ETR}$)  was introduced by Fujimoto over the theory $\mathbf{NBG}$ \cite[Definition 88]{fujimoto} which is equivalent to the system $\Delta^1_0\text{-}\mathbf{TR}$ studied by Sato \cite[Definition 6]{Sato2}. Since we are working in a theory of sets and not classes, the natural principle to consider will be $\Delta_0$ Transfinite Recursion, which we will show to be equivalent over $\B$ to $\text{Bisimulation Axiom}$. 
We follow the convention given in \cite{Sato2}  for transfinite recursion.
    Given a set with relation $\pair{X,\prec}$, where $\prec$ isn't necessarily transitive, and $H\subseteq a\times X$, we write 
    $$H_s= \{u\in a: \pair{u,s}\in  H\}$$
    $$H_{\prec s}= \bigcup_{t\prec s}(H_{t}\times \{t\})$$
    $$X_{\prec s}=\{t\in X:t\prec s\}.$$
Let $u$ and $v$ denote the first and second variables. For a formula $\psi$, in which the variable $v$ only appears in atomic formulas of the form $x\in v$ or $x\notin v$\footnote{This condition is tacitly being assumed when formulating such principle for second order objects in a system like $\mathbf{NBG}$ or $\mathbf{ACA}_0$.} and $n\geq \max FV(\psi)$, by $\psi\text{-}\mathbf{TR}$ we mean the statement, for any sets $a$, $p_0,\dots,p_n$, and well founded relation $\pair{X,\prec}$ there is an $H\subseteq a\times X$ such that
$$\forall s\in X\;\forall c\in a\,(c\in H_s\leftrightarrow  \psi[c,H_{\prec s},p_0,\dots, p_n]).$$
By $\Delta_0\text{-}\mathbf{TR}$ we mean the axiom scheme which consists of all instances $\psi\text{-}\mathbf{TR}$ where $\psi$ is $\Delta_0$ and of the appropriate form. 

\begin{thm}\label{C->ETR} $\C\vdash\Delta_0\text{-}\mathbf{TR}$.
\end{thm}
\begin{proof} 
Let $\psi$ be a $\Delta_0$ formula such that the variable $v$ only appears in atomic formulas of the form $x\in v$ or $x\notin v$. Let $a$ be a set and $\tau$ be a variable assignment such that for all $i\leq n\,\tau(i+2)=p_i$. Assume that $\psi$ and all other formulas in this construction are in negation normal form. For each $s\in X$ we will construct trees $\top^s_\theta$, $\bot^s_\theta$, and $S^s_{\pair{\theta,\sigma}}$ so that
$$(\pi S^s_{\pair{\psi,\tau[c/u]}}=\pi\top_\psi^s)\;\;\leftrightarrow\;\;( \pi S^s_{\pair{\psi,\tau[c/u]}}\neq\pi\bot_\psi^s)\;\;\leftrightarrow\;\;( c\in H_s).$$ 
In this case, we will only need to define such trees for $\theta$ subformula of $\psi$.  Throughout the proof,  $$\sigma\in TC(a\cup (a\times X)\cup\text{rng}(\tau))^{<\omega}$$ will be a sufficiently long variable assignment. We point out that the value $\sigma(v)$ will not matter in the construction of $S^s_{\pair{\theta,\sigma}}$.
\begin{enumerate}
    
    \item If $\theta$ is atomic or negation of an atomic formula which  does not contain the variable $v$, then $\top^s_\theta=\{\emptyset, (0),(0,0),(1)\}$, $\bot^s_\theta=\{\emptyset,(0),(0,0)\}$, and 
    $$S_{\pair{\theta,\sigma}}=\begin{cases}
        \top_{\theta}\;\;\ \text{ if } \pair{\theta,\sigma}\in Th(TC(a\cup(a\times X)),\in) \\
 \bot_{\theta}\;\;\; \text{ otherwise}
 \end{cases}$$
    \item The definition of the trees $\top^s_{\theta_0\wedge \theta_1},\bot^s_{\theta_0\wedge \theta_1},\top^s_{\theta_0\vee \theta_1},\bot^s_{\theta_0\vee \theta_1},\top^s_{\forall x\, \theta},\bot^s_{\forall x\, \theta}$,$\top^s_{\exists x\, \theta}$, $\bot^s_{\exists x\, \theta}$, \\$S^s_{\pair{\theta_0\wedge \theta_1,\sigma}}$,  $S^s_{\pair{\theta_0\vee \theta_1,\sigma}}$, $S^s_{\pair{\forall x\, \theta,\sigma}}$, and $S^s_{\pair{\exists x\, \theta,s}}$ will be the same as in \Cref{thisatrue}.
    
    \item 
    $$\top_{x\in v}^s=\bot^s_{x\notin v}=\bigcup_{t\prec s}(t)^\frown \left(\left( \bigcup_{q\prec s}(\pair{q,0})^\frown \bot^q_\psi \right)\cup (s)^\frown  \top^t_\psi \right)$$
     \[  \top^s_{x\in v}\quad=\quad \raisebox{-0.5\height}{	\begin{tikzpicture}[
		level distance=12mm,
		level 1/.style={sibling distance=40mm},
        level 2/.style={sibling distance=15mm},
		every node/.style={font=\normalsize},
		edge from parent/.style={draw}
		]
		
		\node {$\bullet$}
		child { node {$t$} 
            child { node {$\top^t_\psi$}}
            child { node {$\bot_\psi^q$} }
            child { node {$\cdots$} edge from parent[dotted] 
		node[midway, right, xshift=20pt] {\text{all} $q\prec s$} }}
        child { node {$\cdots$} edge from parent[dotted] 
		node[midway, right, xshift=20pt] {\text{all} $t\prec s$} };
\end{tikzpicture} }   \]
    which has collapse
    $$\{\{\pi \bot^q_\psi:q \prec s\}\cup \{\pi \top^t_\psi\}:t\prec s\}$$
    and
    $$\bot_{x\in v}^s=\top^s_{x\notin v}=\top^s\cup(s)^\frown  \bigcup_{t\prec s}(\pair{t,0})^\frown  \bot^t_\psi $$
     \[  \bot^s_{x\in v}\quad=\quad \raisebox{-0.5\height}{	\begin{tikzpicture}[
		level distance=12mm,
		level 1/.style={sibling distance=40mm},
        level 2/.style={sibling distance=15mm},
		every node/.style={font=\normalsize},
		edge from parent/.style={draw}
		]
		
		\node {$\bullet$}
        child { node{$s$} 
        child { node{$\bot^t_\psi$} }
        child { node {$\cdots$} edge from parent[dotted] 
		node[midway, right, xshift=10pt] {\text{all} $t\prec s$} }
        }
		child { node {$t$} 
            child { node {$\top^t_\psi$}}
            child { node {$\bot^q_\psi$} }
            child { node {$\cdots$} edge from parent[dotted] 
		node[midway, right, xshift=20pt] {\text{all} $q\prec s$} }}
        child { node {$\cdots$} edge from parent[dotted] 
		node[midway, right, xshift=20pt] {\text{all} $t\prec s$} };
\end{tikzpicture} }   \]
    which has collapse
    $$\{\{\pi \bot^q_\psi:q\prec s\}\cup \{\pi \top^t_\psi\}:t\prec s\}\cup\{ \{\pi \bot^t_\psi:t\prec s\}\}$$
\end{enumerate}
For $S^s_{\pair{x\in v,\sigma}}$ and $S^s_{\pair{x\notin v,\sigma}}$, we check if $\sigma(x)$ is a pair of the form $\pair{c,t}\in a\times X_{\prec s}$ if not, set $$S^s_{\pair{x\in v,\sigma}}=S^s_{\pair{x\notin v,\sigma}}=\bot^s_{x\in v}=\top^s_{x\notin v}$$ otherwise let 
    $$S^s_{\pair{x\in v,\sigma}}=S^s_{\pair{x\notin v,\sigma}}=\top^s_{x\in v}\cup(t)^\frown  \left(\bigcup_{q\prec s}(\pair{q,0})^\frown \bot^q_\psi \cup(s)^\frown S^t_{\pair{\psi,\tau[c/u]}}\right)$$
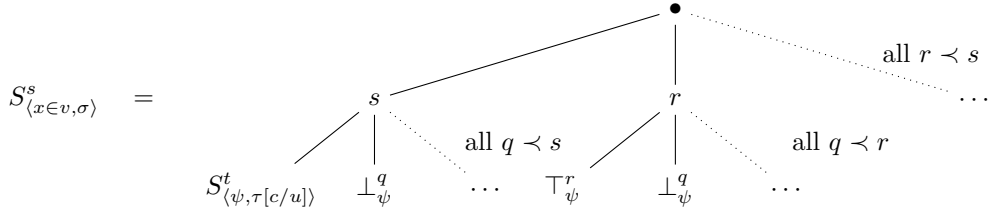
     \[  S^s_{\pair{x\in v,\sigma}}\quad=\quad \raisebox{-0.5\height}{	\begin{tikzpicture}[
		level distance=12mm,
		level 1/.style={sibling distance=40mm},
        level 2/.style={sibling distance=15mm},
		every node/.style={font=\normalsize},
		edge from parent/.style={draw}
		]
		
		\node {$\bullet$}
        child { node{$s$} 
        child { node{$S^t_{\pair{\psi,\tau[c/u]}}$}}
        child { node{$\bot^q_\psi$} }
        child { node {$\cdots$} edge from parent[dotted] 
		node[midway, right, xshift=10pt] {\text{all} $q\prec s$} }
        }
		child { node {$r$} 
            child { node {$\top^r_\psi$}}
            child { node {$\bot^q_\psi$} }
            child { node {$\cdots$} edge from parent[dotted] 
		node[midway, right, xshift=20pt] {\text{all} $q\prec r$} }}
        child { node {$\cdots$} edge from parent[dotted] 
		node[midway, right, xshift=20pt] {\text{all} $r\prec s$} };
\end{tikzpicture} }   \]
    which has collapse the set
    $$\pi\top^s_{x\in v}\cup\{\{\pi \bot^q_\psi:q\prec s\}\cup\{\pi S^t_{\pair{\psi,\tau[c/u]}}\}\}.$$
Define
$$T=\bigcup_{s\in X,\theta\in Form,\sigma\in (a\cup (a\times X))^{<\omega}}\left((\pair{s,\theta,1})^\frown \top^s_\theta \cup (\pair{s,\theta,0})^\frown \bot^s_\theta \cup (\pair{s,\theta,\sigma})^\frown S^s_{\pair{\theta,\sigma}}\right).$$
To verify if a sequence $\sigma$ is in $T$, it suffices to check if it satisfies the rules given by the recursive definitions of $\top^s_\theta$, $\bot^s_\theta$, and $S^s_{\pair{\theta,\sigma}}$. So $T$ is $\Delta_0$-definable relative to a sufficiently large set.
Furthermore, since $X$ is well founded, the tree $T$ will also be well founded.
As in the proof of \Cref{thisatrue} we may identify each of the trees $\top^s_\theta,\bot^s_\theta,$ and $ S^s_{\pair{\theta,\sigma}}$ with a subtree of $T$.
By $\Delta_0$ induction relative to $T$ and its collapsing function, for every formula $\theta$ and $s\in X$ we have $\pi\top^s_{\theta}\neq \pi \bot^s_{\theta}$ and $\pi S^s_{\pair{\theta,\sigma}}\in \{\pi\top^s_\theta,\pi\bot^s_\theta\}$.
Define 
$$H=\{\pair{c,s}\in a\times X: \pi S^s_{\pair{\psi,\tau[c/u]}}=\pi \top^s_{\psi}\}$$
Let $n$ be larger than any index of a variable appearing in $\psi$.
By an external induction on the complexity of $\theta$, subformula of $\psi$, we show that 
$$(\pi S^s_{\pair{\theta,\sigma}}=\pi\top^s_\theta)\;\;\leftrightarrow\;\; \theta[\sigma(0),H_{\prec s},\sigma(2),\dots, \sigma(n)]$$
All cases except for when $\theta$ is of the form $x\in v$ or $x\notin v$ are proved as in \Cref{thisatrue}. For the case $x\in v$, if $\sigma(x)=\pair{c,t}\in a\times X_{\prec s}$ then
$$(\pi S^s_{\pair{x\in v,\sigma}}=\pi\top^s_{x\in v})\;\leftrightarrow\; (\pi S^t_{\pair{\psi,\sigma[c/u]}}=\pi\top^t_\psi)\;\leftrightarrow \;( c\in H_t)\;\leftrightarrow\; (\sigma(x)\in H_{\prec s})$$
where the first equivalence is by construction and the second equivalence is by definition of $H$. Otherwise if $\sigma(x)\notin a\times  X_{\prec s}$ then by construction we have $ S^s_{\pair{x\in v,\sigma}}=\bot^s_{x\in v}$. For the case $x\notin v$, if $\sigma(x)=\pair{c,t}\in a\times X_{\prec s}$ then  $$(\pi S^s_{\pair{x\notin v,\sigma}}=\pi\top^s_{x\notin v})\;\leftrightarrow\; (\pi S^t_{\pair{\psi,\tau[c/u]}}=\pi\bot^t_\psi)\;\leftrightarrow\; ( c\notin H_t)\;\leftrightarrow\; (\sigma(x)\notin H_{\prec s})$$
otherwise if $\sigma(x)\notin a\times  X_{\prec s}$ then by construction we have $ S^s_{\pair{x\notin v,\sigma}}=\bot^s_{x\in v}=\top^s_{x\notin v}$.  So 
$$ H_s=\{c\in a: \pi S^s_{\pair{\psi,\tau[c/u]}}=\pi \top^s_{\psi}\}=\{c\in a: \psi[c,H_{\prec s},p_0,\dots,p_n]\}$$
which proves $H$ is the desired set.
\end{proof}

\begin{cor}
    $\mathbf{B}\vdash \text{Bisimulation Axiom}\leftrightarrow \Delta_0\text{-}\mathbf{TR}$
\end{cor}
\begin{proof}
    $\text{Bisimulation Axiom}\rightarrow \Delta_0\text{-}\mathbf{TR}$ can be shown by adapting the previous proof, where a bisimulation on $T$ will take the role of the collapse function.\\
    \\
    We show $\B\wedge\Delta_0\text{-}\mathbf{TR}\rightarrow \text{Bisimulation Axiom}$. Let $T\subseteq X^{<\omega}$ be a well founded tree, and consider the order $\prec$ on $T\times T$ where $(\sigma_0,\sigma_1)\prec (\tau_0,\tau_1)$ if for each $i\leq 1$, $\sigma_i$ is an immediate predecessor to $\tau_i$. Since $T$ is well founded, so will $\pair{T\times T, \prec}$. Let $\theta$ be the $\Delta_0$ formula with parameter $\pair{T\times T,\prec}$ which states that
    \begin{enumerate}
        \item $u=\pair{\sigma,\tau}\in T\times T$;
        \item $\forall \rho_0\prec \sigma\, \exists \rho_1\prec \tau\, \pair{\pair{\rho_0,\rho_1,},\pair{\rho_0,\rho_1,}}\in v$;
        \item $\forall \rho_1\prec \tau\, \exists\rho_0\prec \sigma\,\pair{\pair{\rho_0,\rho_1,},\pair{\rho_0,\rho_1,}}\in v$.
    \end{enumerate}
     Let $H$ be the set obtained by iterating $\theta$ along $\prec$. The set $\{\pair{\sigma,\tau}:\pair{\pair{\sigma,\tau},\pair{\sigma,\tau}}\in H\}$ will be a bisimulation on $T$.
\end{proof}

\section{Clopen Games}
As mentioned in the introduction, Gitman and Hamkins proved that, over $\mathbf{NBG}$, Clopen Determinacy for class-sized games is equivalent to $\Delta^1_0\text{-}\mathbf{TR}$ \cite[Theorem 9]{hatman}. Similarly, it can be shown over $\B$ that Clopen Determinacy for set-sized games is equivalent to $\mathbf{ETR}$ or $\Delta_0\text{-}\mathbf{TR}$. Gitman and Hamkins' proof can be carried out as is over the system $\B$. So the results in this section cannot be considered new. We will outline an alternative proof that makes use of Stirling's description of bisimulations as winning strategies in a game \cite[Proposition 1]{stirlin}. 
\\
\\
    Over $\B$, a clopen game will be a well founded tree $G\subseteq X^{<\omega}$. A match will be a terminal sequence $\tau\in G$ where if $|\tau|$ is odd , then we say Player $I$ wins, otherwise we say Player $II$ wins. A strategy for Player $I$ is a map which sends every non-terminal sequence in $G$ of even length to one  of its immediate successors, and a strategy for Player $II$ is a map which sends every non-terminal sequence in $G$ of odd length to one of its immediate successors. A strategy for a player is winning if applying that strategy to any sequence of moves from the other player leads to a winning match.
    
    Given a well founded tree $T\subseteq X^{<\omega}$ and a pair of nodes $\pair{\sigma,\tau}\in T^2$, the bisimulation game on $T$ for the pair $\pair{\sigma_0,\tau_0}$ is played as follows. Player $I$ seeks to disprove that $\pair{\sigma,\tau}$ is in the bisimulation, while Player $II$ tries to instead prove that the pair is in the bisimulation. At the start of the turns $2n$ and $2n+1$, there will be in play a pair of sequences $\pair{\sigma_n,\tau_n}$ which were either the sequences played in the previous round or the starting sequences if $n=0$. On turn $2n$, Player $I$ picks an immediate successor of $\sigma_n$ or $\tau_n$ and on turn $2n+1$ Player $II$ responds by playing an immediate successor of the other sequence. Let $\sigma_{n+1}$ and $\tau_{n+1}$ be the extension of $\sigma_n$ and $\tau_n$ respectively which were played on turns $2n$ and $2n+1$. The first player not to be able to make a move loses; that is, Player $I$ wins if at the start of their turn the pair $\pair{\sigma_n,\tau_n}$ contains exactly one terminal and one non-terminal node, while if both sequences are terminal then Player $II$ wins.
    The bisimulation game on $T$ is the game where on the first move Player $I$ plays a pair of nodes $\pair{\sigma_0,\tau_0}$ and Player $II$ chooses whether to play as Player $I$ or Player $II$ in the bisimulation game on $T$ for the pair $\pair{\sigma,\tau}$. Since $T$ is well founded, any match can only have finite length, so the game is clopen.  Player $I$ cannot have a winning strategy for the bisimulation game on $T$ as Player $II$ can choose whether to start first or not. One can verify that, given a strategy for Player $II$, the set 
    $$B=\{\pair{\sigma,\tau}: \text{ Player }II \text{, following the winning strategy, chooses to play second}\}$$
    is a bisimulation on $T$. Formalizing this argument, and the usual proof of Clopen Determinacy from transfinite recursion (See \cite[Theorem 9]{hatman} for the proof over $\mathbf{NBG}$) over $\B$ gives the following result:

\begin{thm}
    Over $\B$, the following are equivalent:
    \begin{enumerate}
        \item Every well founded tree $T$ has a bisimulation.
        \item Clopen determinacy.
        \item $\Delta_0\text{-}\mathbf{TR}$.
    \end{enumerate}
\end{thm}

\section{The constructible universe}

\begin{defin} Given a set $a$, by $\mathrm{Def}(a,\in)$ we mean the set
$$\{c\subseteq a: \exists \theta\in Form\, \exists \sigma\in a^{<\omega}\,(s\in c\;\leftrightarrow\; \pair{\theta,\sigma[s/x]}\in Th(a,\in))\}$$
\end{defin}
We would like to show that $\C\vdash \forall b\, \forall \alpha\in Ord\, ( L_\alpha(b)\text{ exists})$ where by ``$L_\alpha(b)$ exists" we mean 
$$\exists f:\alpha\rightarrow V\,\left(f(0)=TC(b)\,\wedge\forall \beta<\alpha\; f(\beta)=\bigcup_{\gamma<\beta}\mathrm{Def}(f(\gamma),\in)\right)$$
Although we could  use $\Delta_0$ Transfinite Recursion  to define the structure $L_\alpha(b)$ as a well founded relation on a sufficiently large set of terms and then take its collapse, as in \cite{Simpson}, we find it more appropriate to tailor a construction, similar to that in \Cref{thisatrue} and \Cref{C->ETR}, which will make the appropriate membership relation on the set of terms $\Delta_0$ relative to some collapsing function. This will be done by first defining the set of terms $\mathfrak{T}_\alpha$ of the language of ramified set theory of level $<\alpha$ (cf.\ \cite[Definition 1.1]{Buchholz}).  For each $\beta\leq \alpha$, formula $\theta$, and variable assignment $\sigma\in (\mathfrak{T}_\alpha)^{<\omega}$ we define well founded trees $S^\beta_{\pair{\theta,\sigma}}$, $\top^\beta_\theta$, and $\bot^\beta_\theta$  such that  $$ L_\beta(b)\models \theta[t_0,\ldots,t_n]  \;\;\leftrightarrow \;\; \pi S^\beta_{\pair{\theta,(t_0,\ldots,t_n)}}=\pi\top^\beta_\theta \;\;\leftrightarrow \;\; \pi S^\alpha_{\pair{\theta,(t_0,\ldots,t_n)}}\neq \pi\bot^\alpha_\theta.$$ 
The membership relation will then be given by the atomic formulas. 
\begin{defin}
   We define the collection of terms $\mathfrak{T}_\alpha$. The terms of level $-1$ will represent elements of $TC(b)$ and elements of level $\alpha \geq 0$ will represent collections of the form $$\{u\in L_\alpha(b):L_\alpha(b)\vDash\theta[u,t_0,\dots,t_n]\}$$ where $t_0,\dots,t_n$ are terms of level $<\alpha$. Such terms are formally defined as finite trees of sequences.
  Let $u$ denote the variable whose index is $0$. Let $pr_0,pr_1:V^3\rightarrow V$ denote the projection of triples on the first and second coordinate, respectively. For $\alpha\in Ord$ the terms of level $\alpha$ are all finite trees $$t\subseteq (\omega\times(( \alpha \times Form)\cup (\{-1\}\times TC(b)))^{<\omega}$$ satisfying the following.
    \begin{enumerate}
        \item $t$ has a unique sequence $\sigma$ of length $1$; let $r(t)$ denote $\sigma(0)$.
        \item $r(t)$ is a triple of the form $\pair{0,\alpha,\theta}$ where $\theta$ is some formula in the language of set theory.
        \item For all $\sigma \in t$ we have
        $$\forall i<j<|\sigma|\,  (pr_1(\sigma(j))<pr_1(\sigma(i))).$$
        \item $\sigma\in t$ is terminal if $pr_0(\sigma(|\sigma|-1))=-1$.
        \item For all $\sigma\in t$ which are non terminal, $\sigma(|\sigma|-1)$ is of the form $\pair{m,\beta,\psi}$. For all $n\in FV(\psi)\setminus\{0\}$ there exists a unique $\pair{\beta,\xi}\in( \alpha\times Form)\cup (\{-1\}\times TC(b))$ such that $\sigma^\frown(\pair{n,\beta,\xi})\in t$.
    \end{enumerate}
    By the Finite Powerset Axiom and $\Delta_0$ Separation, $\B$ proves that for all $\alpha\in Ord$ the terms of level $<\alpha$ form a set, which we denote as $\mathfrak{T}_\alpha$. We will avoid using the notation given by the formal definition of $\mathfrak{T}_\alpha$. We will identify a term $t=\{\emptyset,(\pair{0,-1,c})\}$ of level $-1$ with the element $c\in TC(b)$, and a term $t\in \mathfrak{T}_\alpha$ of level $\geq 0$ will be written as
    $$\{u\in L_\beta(b):L_\beta(b)\vDash \theta[u,t_0,t_1,\dots, t_n]\}$$
    where $r(t)=\pair{0,\beta,\theta}$ and  $t=(r(t))^\frown\bigcup_{i\leq n} t_i^*$, where $t^*_i$ is obtained by replacing $r(t_i)=\pair{0,\xi_i,\beta_i}$ with $\pair{i+1,\xi_i,\beta_i}$.
    We will need a term for each $L_\beta(b)$, such as $\{\emptyset,(\pair{0,\beta,v=v})\}$, which, with an abuse of notation, we will denote simply as $ L_\beta(b)$. We consider elements of $Val_\alpha=(\mathfrak{T}_\alpha)^{<\omega}$ as assigning terms of level $<\alpha$ to variables.
    \end{defin}
\begin{thm}
   $\C\vdash \forall a\, \forall \alpha\in Ord \;(L_\alpha(b)\text{ exists })$ 
\end{thm}
\begin{proof}
We outline how to construct for each $\beta\leq \alpha$ and each formula $\theta$ the trees $\top^\beta_\theta$, $\bot^{\beta}_\theta$, and $S^\beta_{\pair{\theta,\sigma}}$. As in the proof of \Cref{thisatrue} and \Cref{C->ETR}, we will make use of a finite set of rules which will depend on the complexity of $\theta$. From the collapsing function of such trees, we can define a relation $\cin$ on the collection of terms $\mathfrak{T}_\alpha$ such that the collapse of $\cin|_{\mathfrak{T}_\alpha}$ will be $L_\alpha(b)$. To simplify the construction and avoid redundancies, we consider the language of set theory where equality is not part of the language but rather is the following abbreviation
$$y=z\;\equiv \;\forall x\;(x\in y\leftrightarrow x\in z).$$
For the atomic case, we use the following one-size-fits-all trees $\top^\beta$ and $\bot^\beta$ which are defined as
$$\top^\beta=\bigcup_{\varphi\in Form,\gamma<\beta} (\pair{\varphi,\gamma})^\frown\left(\left( \bigcup_{\pair{\psi,\delta}\in Form\times \beta} (\pair{\psi,\delta})^\frown \bot^\delta_\psi\right)\cup (\pair{\varphi,\top})^\frown \top^\gamma_\varphi\right)$$
and
$$\bot^\beta=\top^\beta\cup (\bot)^\frown \bigcup_{\pair{\psi,\gamma}\in Form\times \beta}(\pair{\psi,\gamma})^\frown \bot^\gamma_\psi$$
We define $\top^\beta_\theta$ and $\bot^\beta_\theta$ based on the complexity of the formula. If $\theta$ is atomic or the negation of an atomic, then $\top^\beta_\theta=\top^\beta$ and $\bot^\beta_\theta=\bot^\beta$. For the conjunction, disjunction, and quantifier cases, the construction is the same as in \Cref{thisatrue}. By induction on $\beta\leq \alpha$ and on the complexity of $\theta$ and using \Cref{importantlater} we have $\pi\top^\beta_\theta\neq \pi\bot^\gamma_\psi$ for every formula $\psi$ and $\gamma\leq \alpha$. This ensures that for all $\beta\leq \alpha$ $\pi\top^\beta\neq \pi\bot^\beta$.
We now are left to define $S^\beta_{\pair{\varphi,\sigma}}$, where $\sigma$ assigns terms of level $<\beta$ to the variables.

\begin{enumerate}
    \item In the case where $\beta=0$, if $\theta\equiv x_i\in x_j$ then $$S^0_{\pair{\theta,\sigma}}=\begin{cases}
        \{ \emptyset,(0),(0,0),(1)\}\;\; \text{ if }\;\; \sigma(i)\in \sigma(j)\\
 \{\emptyset,(0),(0,0)\}\quad\quad \text{ otherwise}
 \end{cases}$$
and  if $\theta\equiv x_i\notin x_j$ then $$S^0_{\pair{\theta,\sigma}}=\begin{cases}
        \{ \emptyset,(0),(0,0),(1)\}\;\; \text{ if }\;\; \sigma(i)\notin \sigma(j)\\
 \{\emptyset,(0),(0,0)\}\quad\quad \text{ otherwise}
 \end{cases}$$
 \item In the case where $\beta>0$, and $\theta\equiv x_{j_0}\in x_{j_1}$ or $\theta\equiv x_{j_0}\notin x_{j_1}$, for $i\leq 1$ let $$t_i=\sigma(j_i)=\{u\in L_{\eta_i}(b): L_{\eta_i}(b)\vDash\varphi_i[u,t^i_0,\dots,t^i_{m_i}]\}.$$

 \begin{enumerate}
     \item If $\eta_0<\eta_1$ then let $\delta=\eta_1$ and $\pair{\rho,\sigma^*}$ be $\varphi_1[t_0, t^1_0,\dots,t^1_{m_1}]$.
     \item If $\eta_0\geq \eta_1\geq 0$ then let $\delta=\eta_0$ and $\pair{\rho,\sigma^*}$ be
     $$\exists x\in  L_{\eta_1}(b)\,(\forall s\,(s\in x\leftrightarrow \varphi_0[s,t^0_0,\dots,t^0_{m_0}]) \wedge \varphi_1(x,t^1_0,\dots, t^1_{m_1}])$$ 
     \item If $\eta_0> \eta_1=-1$ then let $\delta=\eta_0$ and $\pair{\rho,\sigma^*}$ be
     $$\exists x\in  L_0(b)\,(\forall s\,(s\in x\;\leftrightarrow\; \varphi_0[s,t^0_0,\dots,t^0_{m_0}]) \wedge x\in t_1)$$ 
     \item If $\eta_0= \eta_1=-1$ then let $\pair{\rho,\sigma^*}$ be
     $t_0\in t_1$ and $\delta=0$
 \end{enumerate}
 If $\theta\equiv x_i\in x_j$ then let $\theta^*$ be the negation normal form of $\rho$ and if $\theta\equiv x_i \notin x_j$ let $\theta^*$ be the negation normal form of $\neg\rho$.
 If $\beta=\gamma+1$ then
 $$S^\beta_{\pair{\theta,\sigma}}=\top^\gamma \cup(\pair{\theta,\sigma,\beta})^\frown\left( \bigcup_{\pair{\psi,\gamma}\in  Form\times \beta}(\psi)^\frown \bot^\gamma_\psi\right)\cup  (\pair{\theta^*,\sigma^*})^\frown S^\delta_{\pair{\theta^*,\sigma^*}}.$$
 We note that by construction $(\pi S^\beta_{\pair{\theta,\sigma}}=\pi\top^\beta)\;\; \leftrightarrow\;\; (\pi S^\rho_{\pair{\theta^*,\sigma^*}}=\pi \top^\delta_{\theta^*})$.
 \item $S_{\pair{\theta_0\wedge \theta_1,\sigma}}^{\beta}$ is the union of the following 
  \begin{enumerate}
        \item $(\pair{\theta_0\wedge\theta_1,\beta,\sigma},\pair{1,1})^\frown OPair(S^{\beta}_{\pair{\theta_0,\sigma}},S^{\beta}_{\pair{\theta_1,\sigma}})$
        \item $(\pair{\theta_0\wedge\theta_1,\beta,\sigma},\pair{1,0})^\frown OPair(\top^\beta_{\theta_0},\bot^\beta_{\theta_1})$
        \item $(\pair{\theta_0\wedge\theta_1,\beta,\sigma},\pair{0,1})^\frown OPair(\bot^\beta_{\theta_0},\top^\beta_{\theta_1})$
        \item $(\pair{\theta_0\wedge\theta_1,\beta,\sigma},\pair{0,0})^\frown OPair(\bot^\beta_{\theta_0},\bot^\beta_{\theta_1})$
    \end{enumerate}
\item $S_{\pair{\theta_0\vee \theta_1,\sigma}}^{\beta}$ is the union of the following
    \begin{enumerate}
        \item $(\pair{\theta_0\vee\theta_1,\beta,\sigma},\pair{1,1})^\frown OPair(\top^\beta_{\theta_0},\top^\beta_{\theta_1})$
        \item $(\pair{\theta_0\vee\theta_1,\beta,\sigma},\pair{1,0})^\frown OPair(\top^\beta_{\theta_0},\bot^\beta_{\theta_1})$
        \item $(\pair{\theta_0\vee\theta_1,\beta,\sigma},\pair{0,1})^\frown OPair(\bot^\beta_{\theta_0},\top^\beta_{\theta_1})$  
        \item $(\pair{\theta_0\vee\theta_1,\beta,\sigma},\pair{0,0})^\frown OPair(S^{\beta}_{\pair{\theta_0,\sigma}},S^{\beta}_{\pair{\theta_1,\sigma}})$
    \end{enumerate}
\item $$S^{\beta}_{\pair{ \forall x\,\theta,\sigma}}= (\pair{\forall x\,\theta,\beta,\sigma},1)^\frown \top^\beta_{\theta}\cup \bigcup_{t\in \mathfrak T_\beta} (\pair{\forall x\,\theta,\beta,\sigma},t)^\frown S^{\beta}_{\pair{\theta,\sigma[t/x]}}$$
\item $$S^{\beta}_{\pair{ \exists  x\,\theta,\sigma}}= (\pair{\exists x\,\theta,\beta,\sigma},0)^\frown \bot^\beta_{\theta}\cup\bigcup_{t\in \mathfrak T_\beta} (\pair{\exists x\,\theta,\beta,\sigma},t)^\frown S^{\beta}_{\pair{\theta,\sigma[t/x]}}$$
\end{enumerate}
Take $T$ sufficiently large as to contain all $\beta<\alpha$, $S^\beta_{\pair{\theta,\sigma}},\top^\beta_\theta,$ and $\bot^\beta_\theta$ as subtrees and let $\pi$ be the collapsing function of $T$. Recall that $\mathfrak{T}_\alpha$ denote the terms of level $< \alpha$ we define $\overset{\circ}{\in}$ to be all pairs $t_0$ and $t_1$ such that $t_0$ has level strictly less than $t_1$, or both have level $-1$, and $\pi S^\alpha_{\pair{x_0\in x_1, (t_0,t_1)}}=\pi \top^\alpha_{x_0\in x_1}$. We take the collapse of $(\mathfrak{T}_\alpha,\overset{\circ}{\in})$ and use $\pi$ to denote the collapsing function on $T$ and on $(\mathfrak{T}_\alpha,\cin)$ as this will not cause ambiguity. We show that the map $\beta\mapsto \pi\mathfrak T_\beta$ witnesses that $\pi\mathfrak T_\alpha=L_\alpha(b)$. \\
\\
To verify $\pi\mathfrak{T}_0=TC(b)$, let $t_0$ and $t_1$ be terms of level $-1$, which by convention we identify with elements of $TC(b)$. We have the following equivalences 
$$t_0\cin t_1\;\leftrightarrow\;( \pi S^\alpha_{\pair{x_0\in x_1,(t_0,t_1)}}=\pi\top^\alpha_{x\in y})\;\leftrightarrow\;( \pi S^0_{\pair{x_0\in x_1,(t_0,t_1)}}=\pi \top^0_{x_0\in x_1})\;\leftrightarrow\; t_0\in t_1 $$
The first equivalence is by definition of $\cin$, the second and third by construction. This implies $\pi\mathfrak{T}_0=TC(b)$.
\\
\\
We show that for any $\beta<\alpha$ that 
$$\pi\mathfrak{T}_{\beta+1}=\mathrm{Def}(\pi\mathfrak{T}_\beta,\in)$$
We show that for all $r= \{u\in L_\beta(b):L_\beta(b)\vDash \varphi[u,t_0,\dots, t_n]\}$ we have
$$\forall t\in \mathfrak{T}_\beta\;(\varphi[\pi t,\pi t_0,\pi t_1,\dots,\pi t_n]\in Th(\pi \mathfrak{T}_\beta,\in)\;\leftrightarrow \;\pi t\in \pi r)$$
We prove by induction on $\gamma<\alpha$ and on the complexity of the formula that for all $\varphi$ such that $FV(\varphi)\subseteq n$
$$\forall \sigma\in Val_\gamma\;(|\sigma|>n\rightarrow \,(\pair{\varphi,\pi\circ\sigma}\in Th(\pi\mathfrak{T}_\gamma,\in)\;\leftrightarrow\; \pi S^\gamma_{\pair{\varphi,\sigma}}=\pi\top^\gamma_\varphi))$$
the proof is by induction on the complexity of the formula. The cases for boolean connectives and quantifiers are treated the same way as in \Cref{thisatrue}, which leaves the atomic case. Let $t_0$ and $t_1$ be terms of level $\leq \beta$, for each $i\leq 1$ let 
$$t_i=\{u\in L_{\eta_i}(b):L_{\eta_i}(b)\vDash\varphi_i[u,t^i_0,\dots, t^i_{m_i}]\}$$
We consider the case where $\eta_0\geq \eta_1\geq 0$; the other cases are treated similarly or are simpler. Define
$$\chi[r]\equiv \varphi_1[r,t^1_0,\dots, t^1_{m_1}]$$
and
$$\pair{\psi,\sigma}\equiv \exists x\in  L_{\eta_1}(b)\;(\forall s( s\in x\leftrightarrow \varphi_0[s,t^0_0,\dots, t^0_{m_0}])\wedge \varphi_1(x,t^1_0,\dots, t^1_{m_1}]). $$
The following are equivalent:
\begin{enumerate}
    \item $\pair{x_0\in x_1,(\pi t_0,\pi  t_1)}\in Th(\pi\mathfrak{T}_\beta,\in)$.
    \item $\pi t_0\in \pi t_1$.
    \item $\pi t_0\in \{\pi r: r\cin t_1\}$.
    \item $\exists r\in \mathfrak{T}_{\eta_1}\, (r\cin t_1\wedge \pi r=\pi t_0)$.
    \item $\exists r\in \mathfrak{T}_{\eta_1}\, ( \pi S_{r\in t_1}^{\alpha}=\pi\top_{x\in y}^{\alpha}\wedge \pi r=\pi t_0)$.
    \item  $\exists r\in \mathfrak{T}_{\eta_1}\, ( \pi S_{\chi[r]}^{\eta_1}=\pi\top_{\chi}^{\eta_1}\wedge \pi r=\pi t_0)$.
    \item $\exists r\in \mathfrak{T}_{\eta_1}\,  (\chi[\pi r]\in Th(\pi \mathfrak{T}_{\eta_1},\in)\wedge \pi r=\pi t_0)$.
    \item $\exists x\in \pi\mathfrak{T}_{\eta_1}\,  (\chi(x]\in Th(\pi \mathfrak{T}_{\eta_1},\in)\wedge \forall s(s\in x\leftrightarrow s\in \pi t_0))$.
    \item $\pair{\psi,\pi \circ \sigma}\in Th(\pi\mathfrak{T}_{\eta_1},\in)$.
    \item $\pi S^{\eta_1}_{\pair{\psi, \pi \circ \sigma}}=\pi \top^{\eta_1}_{\psi}$.
    \item $\pi S^\beta_{\pair{x_0\in x_1,( t_0, t_1)}}=\pi \top^\beta_{x\in y}$.
\end{enumerate}
The equivalence $8\leftrightarrow 9$ follows from Tarski conditions of $Th(\pi \mathfrak{T}_{\eta_1},\in)$
and the inductive hypothesis which implies  $\pi L_{\eta_1}(b)=\pi \mathfrak{T}_{\eta_1}$.
\end{proof}

\begin{prop}[Basic facts about $L$] $\C$ proves the following
\begin{enumerate}
    \item For every $b$ and $\alpha\in Ord$ the set $L_\alpha(b)$ is transitive.
    \item If $\alpha$ is a limit ordinal, then $L_\alpha\vDash \B$.
    \item $\forall b,c\,\forall \alpha\in Ord\;(b\subseteq c\rightarrow L_\alpha(b)\subseteq L_\alpha(c))$.
    \item $L_\beta(L_\alpha(b))=L_{\alpha+\beta}(b)$.
    \item $L$ has a definable well ordering and at each limit ordinal $\alpha$ we have $L_\alpha$ has a definable well ordering. Similarly if $a\subseteq L_{\beta}$ then $L(b)$ and $L_\alpha(b)$, for $\alpha>\beta$ limit ordinal, will have definable well ordering.
\end{enumerate}
\end{prop}

\section{Primitive recursive set functions} \label{sec pr}
We wish to show that $\C$ can prove the universe is closed under the primitive recursive set functions.  More precisely, we show  within $\C$ that every primitive recursive set function is stabilized by the map $\alpha \mapsto \varphi(i,\alpha)$ for some $i \in \omega$ (cf.\ Cantini \cite{cantini} and Avigad \cite{avigad}). The fact that primitive recursive branches of the Veblen function bound the primitive recursive ordinal functions, and in particular that every $\vp(\omega,\gamma)$ is closed under such functions,  was  first noticed, at least in print,   by Sch{\"u}tte \cite{schutte}.  

Our result appears to follow from  Simpson's proof sketch of \cite[Theorem 3.1]{simpson82}. Simpson claims that a theory which is equivalent to $\C$ proves the totality of the primitive recursive set functions.  It is worth noting that Simpson includes  an axiom postulating  the existence of $L_\alpha(b)$ for every transitive set $b$ and every ordinal $\alpha$, whereas the results of the previous section show that this is not necessary. The  strategy outlined   in \cite{simpson82}
is to first establish the result for the primitive recursive ordinal functions, and then to formalize the original Jensen–Karp stability theorem. The authors of this paper, however, admit that it is not fully clear how the first part of this plan is intended to be carried out. In any case,  our approach via the Veblen hierarchy provides a method to implement both steps at once.

\begin{defin} The class of primitive recursive set functions is the smallest class $Prim$ such that: 
\begin{enumerate}
    \item The projections $x_0,\dots,x_n\mapsto x_i$, where $n\in \omega$ and $i\le n$, are in $Prim$.
    \item The zero functions $x_0,\dots, x_n\mapsto 0$ are in $Prim$.
    \item Adjunction $x,y\mapsto x\cup \{y\}$ is in $Prim$.
    \item Case distinction $x,y,u,v\mapsto \begin{cases}
        x \text{ if } u\in v,  \\
        y \text{ if } u\notin v,
    \end{cases}$ is in $Prim$.
    \item Composition: if $g_0,\ldots,g_k$ are in $Prim$, then $f(\vec x)=g_0( g_1 (\vec x),\dots,g_k(\vec x))$ is in $Prim$.
    \item Primitive Recursion: if $h$ is in $Prim$, so is $f(x,\vec y)=h(\bigcup\{ f(z,\vec y):z\in x \},x,\vec y)$.
\end{enumerate}
We note that every recursive definition of a primitive recursive set function is $\Sigma$. We will identify $Prim$ functions with their recursive definition.
\end{defin}

\begin{defin}\label{rud}
An important subclass of the primitive recursive set functions is the collection of rudimentary functions, which are  obtained  from: 
\begin{enumerate}
    \item Projections $x_0,\dots,x_n\mapsto x_i$;
    \item Differences $x_0,\ldots,x_n\mapsto x_i\setminus x_j$;
    \item Pairs  $x_0,\ldots,x_n\mapsto \{ x_i, x_j\}$,
\end{enumerate}
and applications of the following rules:
\begin{enumerate}
    \item Composition;
    \item The \lq\lq Union of Image\rq\rq\  scheme:  $f(x,\vec y)=\bigcup\{ g(z,\vec y): z\in x \}$.
\end{enumerate}
By \cite[Lemma 1.8.4]{jensen}, the rudimentary functions are exactly the ones  obtained by composition from the following finite family of functions:
\begin{itemize}
    \item  $F_0(x,y)=\{x,y\}$;
    \item $F_1(x,y)=x\setminus y$;
    \item  $F_2(x,y)=x\times y$;
    \item  $F_3(x,y)=\{\pair{u,w,v}:w\in x\wedge\pair{u,v}\in y\}$;
    \item  $F_4(x,y)=\{\pair{u,v,w}: w\in x\wedge\pair{u,v}\in y\}$;
    \item $F_5(x,y)=\bigcup x$;
    \item $F_6(x,y)=\text{dom}(x)=\{u\in \bigcup^2x:\exists v\in \bigcup^2x\,\pair{u,v}\in x\}$;
    \item $F_7(x,y)=\{\pair{u,v}\in x^2:u\in v\}$;
	\item $F_8(x,y)=\{\{w:\pair{w,z}\in x\}:z\in y\}.$
\end{itemize}
\end{defin}

\begin{defin}
    The system $\mathbf{PRS}\omega$, introduced by Rathjen \cite[Section 6]{R92}, is formalized in the language with membership $\in $ and a function symbol for every recursive definition for a primitive recursive set function. As axioms it has Extensionality, Pair, Union, Regularity, Infinity,  $\Delta_0$ Separation, and for each function symbol its respective recursive definition. 
\end{defin}
\begin{thm}
    $\B$ is a subtheory of $\mathbf{PRS}\omega$.
\end{thm}
\begin{proof}
    An analysis of the axioms of $\B$ and $\mathbf{PRS}\omega$ shows that it suffices to verify that $\mathbf{PRS}\omega$ proves the existence of the transitive closure and the finite powerset. It is a known fact (see for example \cite[Example 0.1]{Mathias2015}) that the transitive closure is rudimentary as it is given by
    $$TC(x)=\bigcup\{TC(y): y\in x\}\cup x.$$
    For the finite powerset,  we need $\omega$ as a constant. The function $$f(x,y)=x\cup\{ u\cup \{ v\}:u\in x\wedge v\in y\}$$ is primitive recursive. Let $h$ be given by the recursive definition 
    $$h(x,y,z)=f(\bigcup\{h(x,y,w):w\in z\},x)$$
    one verifies that $\mathcal{P}_\text{fin}(x)=h(x,\{\emptyset\},\omega)$.
\end{proof}

\subsection{Veblen Hierarchy}

\begin{defin}
    The binary Veblen function $\varphi:Ord\times Ord\rightarrow Ord$ is the function satisfying the following recursive definition
    \begin{enumerate}
        \item $\varphi(0,\alpha)=\omega^\alpha$.
        \item For $\beta>0$, the map $\alpha\mapsto \varphi(\beta,\alpha)$ is the unique map that increasingly enumerates the set
        $$\{\alpha\in Ord: \forall \gamma<\beta\,\,\varphi(\gamma,\alpha)=\alpha\}.$$
    \end{enumerate}

The ordinal $\vp(\alpha,\beta)$ is the least closed under $0$, $\vp(\alpha,\gamma)$ for $\gamma<\beta$, addition and $\gamma\mapsto \vp(\delta,\gamma)$ for $\delta<\alpha$.
    
\end{defin}
\begin{defin}
    We say $\pair{\Lambda,\leq_{\Lambda}}$ is a quasi linear order if $\leq_{\Lambda}$ is reflexive, transitive, and every two elements are comparable. For $\lambda_0,\lambda_1\in \Lambda$ we write $\lambda_0\equiv_{\Lambda}\lambda_1$ to mean $\lambda_0\leq_{\Lambda}\lambda_1\wedge \lambda_1\leq_\Lambda \lambda_0$ and  $\lambda_0<_{\Lambda}\lambda_1$ to mean $\lambda_0\leq_{\Lambda}\lambda_1\wedge \lambda_1\nleq_\Lambda \lambda_0$. We use $+\infty$ and $-\infty $ to denote new elements that satisfy $\forall \lambda\in \Lambda\,(-\infty<_\Lambda\lambda<_\Lambda+\infty)$. By $\Lambda\vert_{<\lambda}$ we mean the set $\{\mu\in \Lambda:\mu<_{\Lambda} \lambda\}$. A sequence $f:\omega \rightarrow \Lambda$ is said to be descending if $\forall i\in \omega\, (f(i+1)<_\Lambda f(i))$.
\end{defin}

\begin{defin} We define a relativized notation system for the Veblen hierarchy similar to the one defined in \cite[Definition 2.7]{Martalban} and \cite[Definition 2.2]{RathjenWeiermann}. Our system of terms will avoid having to define normal forms for its construction, but it has the disadvantage of not having an antisymmetric order. Since our goal is to eventually take the collapse, this will not cause issues. Given a quasi linear order $\Lambda$ and an ordinal $\alpha$, we define $O(\alpha,\Lambda)$ to be all terms generated by the following rules. 
\begin{enumerate}
    \item $0\in O(\alpha,\Lambda).$ 
    \item $ x \in \Lambda\rightarrow \varphi(\alpha,x)\in O(\alpha,\Lambda).$
    \item If $t_1,t_2,\dots,t_n\in O(\alpha,\Lambda)$  and each $t_i$ is not itself sum of terms, then $ t_1+\dots +t_n\in O(\alpha,\Lambda)$.
    \item $\beta< \alpha\wedge t\in O(\alpha,\Lambda)\rightarrow \varphi(\beta,t)\in O(\alpha,\Lambda).$
\end{enumerate}
\noindent
Let $h_\Lambda (t)$ denote the largest $x\in \Lambda$ such that $\varphi(\alpha,x)$ appears in $t$ and $0$ if there is no such $x$. The relation  $t\leq_{\alpha,\Lambda}s $ is given by recursion on the complexity of the terms and is defined by the following rules: 
\begin{enumerate}
    \item $t=0$
    \item $t=\varphi(\alpha,x)\wedge h_{\Lambda}(s)\neq 0\wedge x\leq_{\Lambda} h_\Lambda(s)$
    \item If $t=t_0+\dots +t_n$ and $s=s_0+\dots +s_m$ then there exists a weakly increasing function $f:n+1\rightarrow m+1$ such that $\forall i\leq n \,t_i\leq_{\alpha,\Lambda} s_{f(i)}$ and $s_{f(i)}\leq_{\alpha,\Lambda} t_{i}$ then $f(i+1)>f(i)$ ($n$ and $m$ may be $0$).
    \item If $t=\varphi(\beta,t_0)$ and $s=\varphi(\gamma,s_0)$ we require
    \begin{enumerate}
        \item $\beta<\gamma\rightarrow t_0\leq_{\alpha,\Lambda} s$.
        \item $\beta =\gamma\rightarrow t_0\leq_{\alpha,\Lambda} s_0$.
        \item $\beta> \gamma\rightarrow t\leq_{\alpha,\Lambda} s_0$
    \end{enumerate}
\end{enumerate}
Condition $3$ reflects how sums of additively indecomposable ordinals compare. 

 Formally $O(\alpha,\Lambda)$ will be a subset of $ ((\alpha+1)\times \{0\}\cup \Lambda\times \{1\}\cup5\times \{2\})^{<\omega}$,  where the elements in $5\times \{2\}$ will represent $+,\varphi$, the left and right parenthesis, and the comma. This ensures that:
    \begin{enumerate}
     \item $\B$ proves that for every quasi linear order $\Lambda$ and ordinal $\alpha$, the pair $\pair{O(\alpha,\Lambda),\leq_{\alpha,\Lambda}}$ exists. Furthermore, if $\Lambda$ is well orderable as a set, then  $O(\alpha,\Lambda)$ is also well orderable.
    \item  $\C$ proves for all $\alpha,\beta\in Ord$ $\pair{O(\alpha,\beta),<_{\alpha,\beta}}\in L_{\max\{\alpha,\beta\}+\omega}$.
    \end{enumerate}

\end{defin}
\begin{lemma}
    The system $\B$ proves the following:
    \begin{enumerate}
        \item The relation $\leq_{\alpha,\Lambda}$ on $O(\alpha,\Lambda)$ is a quasi linear order.
        \item If $s_0,s_1<_{\alpha,\Lambda} \varphi(\beta,t)$ then $s_0+s_1<_{\alpha,\Lambda}\varphi(\beta,t)$.
        \item For $\gamma<\beta\leq\alpha$ we have $\varphi(\gamma,\varphi(\beta,t))\equiv_{\alpha,\Lambda} \varphi(\beta,t)$.
        \item If $s<_{\alpha,\Lambda}t$ then $\varphi(\beta,s)<_{\alpha,\Lambda} \varphi(\beta,t)$.
        \item $O(\alpha,\Lambda)\vert_{<\varphi(\alpha,x)}= O(\alpha,\Lambda\vert_{<x} )$.
        \item $O(\alpha,\Lambda)\vert_{<\varphi(\beta,\varphi(\alpha,x))}\cong O(\beta, O(\alpha,\Lambda\vert_{<x} ))$ and such isomorphism is $\Delta_0$-definable relative the set of finite sequences of the two sets.
        
    \end{enumerate}
\end{lemma}

\begin{lemma}\label{HIRST}
    $\B$ proves that for any quasi linear order $\Lambda$, if $\pair{O(0,\Lambda),<_{0,\Lambda}}$ has a descending sequence so will $\Lambda$.
\end{lemma}
\begin{proof}
Every $t\in O(0,\Lambda)$ is equivalent to a term in normal form, that is, there exists a unique finite sequence $\lambda_0\geq_\Lambda\dots\geq_\Lambda \lambda_n$ such that 
    $$t\equiv_{0,\Lambda} \varphi(0,\lambda_0)+\dots +\varphi(0,\lambda_n).$$
    Let $h_i(t)$ be $\lambda_i$ if $i\leq n$ and $-\infty$ otherwise.
    Let $f:\omega\rightarrow O(0,\Lambda)$ be $<_{0,\Lambda}$ descending. Note that each $t$ has a unique normal form and map which sends $t$ to its normal form is $\Delta_0$ relative to $O(0,\Lambda)^{<\omega}$ and $<_{0,\Lambda}$. So we may assume that for all $i\in \omega$ the term $f(i)$ is in normal form. If there exists some $i\in \omega$ such that $h_i(f(n))$ is not eventually constant, the for the least such $i$ the sequence $h_i(f(n))$ will eventually be descending in $\Lambda$. For the case where for each $i$ the sequence $h_i(f(n))$ eventually stabilizes, let $n_i$ to be the least $n$ such that
    $$\forall \,j\leq i\,( n>n_j\wedge\forall m\geq n\;(h_j(f(m))=h_j(f(n)))).$$ Since $f$ was $<_{0,\Lambda}$ descending, the sequence $i\mapsto h_i(f(n_i))$ is $<_\Lambda$ descending.
\end{proof}
\begin{obs}\label{thobs}
    Since for any limit ordinal $\delta$ we have $L_\delta(u)\vDash \B$, if $\pair{\Lambda,<_\Lambda},f\in L_\delta(u)$, where $\Lambda$ is a quasi linear order and $f$ is a descending sequence in $O(0,\Lambda)$, then there is a $<_\Lambda$ descending sequence in $L_{\delta}(u)$.
\end{obs}
\begin{cor}
     $\C$ proves the function $\alpha\mapsto \omega^\alpha$ is total. 
\end{cor}
\begin{proof}
    The set $O(0,\alpha)$ is well orderable, so $<_{0,\alpha}$ is well founded since $\alpha$ does not have any descending sequences. By the properties of  $\pair{O(0,\alpha),<_{0,\alpha}}$ its collapse will be $\omega^\alpha$.
\end{proof}

\begin{lemma} $\C$ proves the binary Veblen function is total.
\end{lemma}
\begin{proof} We follow the proof carried out by Montalbán and Marcone \cite[Theorem 3.6]{Martalban} where the set $L_{\omega^{1+\alpha}}(u)$ will take the role of the $\omega^\alpha$-iteration of the Turing jump. 
Given ordinals $\alpha$ and $\beta$ we wish to show that $\pair{O(\alpha,\beta),<_{\alpha,\beta}}$ is well founded.  By the previous lemma, we may assume $\alpha>0$. The set $O(\alpha,\beta)$ is well orderable, so  $<_{\alpha,\beta}$ is well founded if and only if it does not have  descending sequences. Seeking a contradiction, assume that $f:\omega \rightarrow O(\alpha,\beta)$ is descending. Let $u= L_{\max\{\alpha,\beta\}+\omega}(f)$, by construction $\pair{O(\alpha,\beta),<_{\alpha,\beta}}\in u$ and $f\subseteq u$, therefore $$L_{\omega^{1+\alpha}}(u)=L_{\max\{\alpha,\beta\}+\omega+\omega^{1+\alpha}}(f)$$ admits a definable well ordering. We show that there exists in $L_{\omega^{1+\alpha}+1}(u)$ an infinite descending sequence in the ordinal $\beta$ which would provide the desired contradiction.

It suffices to show by induction on $\gamma\leq \alpha$ that for every $\delta<\omega^{1+\alpha}$ and $\Lambda,g\in L_{\delta}(u)$, where $\Lambda$ is a quasi linear order and $g$ is a descending sequence in $O(\gamma,\Lambda)$, there is a descending sequence $g^*:\omega\rightarrow \Lambda$ in $L_{\delta+\omega^{1+\gamma}+1}(u)$. The case $\gamma=0$ follows from \Cref{thobs}, so we will consider $\gamma>0$. Assume that for all $\xi< \gamma$ the inductive hypothesis holds. Fix $\delta< \omega^{1+\alpha}$, a quasi linear order $\Lambda\in L_\delta(u)$, and a $<_{\gamma,\Lambda}$ descending sequence $g\in L_{\delta}(u)$. By recursion, we define a descending sequence  $(x_n)_{n\in \omega}$ in $\Lambda\cup\{+\infty\}$, let $g_0=g$ and $x_{0}=+\infty$. Assume that we have defines $x_m$ for all $m\leq n$ and $g_n$ which is descending in $O(\gamma,\Lambda\vert_{<x_{n}})$  such that for some $\eta<\omega^{1+\gamma}$ $g_n\in L_{\delta+\eta}(u)$. Let $x_{n+1}=h_{\Lambda}(g_n(0))$. There exists $\xi<\gamma$ and $m\in \omega$ such that
$$\varphi(\gamma,x_{n+1})\leq_{\gamma,\Lambda} g_n(0)<_{\gamma,\Lambda} \varphi^m(\xi,\varphi(\gamma, x_{n+1})+1) $$
By the isomorphism
$$O(\gamma,\Lambda)\vert_{< \varphi^m(\xi,\varphi(\gamma,x_{n+1})+1)}\cong O^m(\xi, O(\gamma,\Lambda\vert_{< x_{n+1}} )+1)$$
we may identify $g_n$ as being a descending sequence in $O^m(\xi,O(\gamma,\Lambda\vert_{<x_{n+1}})+1)$. By the inductive hypothesis applied $m$ times, this implies that there is a sequence 
$$h\in L_{\delta+\eta+\omega^{1+\xi}\cdot m+1}(u)\; \text{ descending in }\;O(\gamma, \Lambda\vert_{<x_{n+1}}).$$ Let $g_{n+1}$ be the least such $h$ with respect to the definable well ordering of $L_{\omega^\alpha}(u)$. By indecomposibility of $\omega^{1+\gamma}$ we have $\eta+\omega^{1+\xi}\cdot m+1<\omega^{1+\gamma}$. This shows that the sequence $(x_n)_{n\in \omega}$ is well defined and by construction,  $(x_{n+1})_{n\in \omega}$ is descending in $\Lambda$. Since it is definable in $L_{\delta+\omega^{1+\gamma}}(u)$  the sequence $(x_{n+1})_{n\in \omega}$ will be a member of $L_{\delta+\omega^{1+\gamma}+1}(u)$.

So $\C$ proves that for all $\alpha,\beta\in Ord$ the quasi linear order $\pair{O(\alpha,\beta)<_{\alpha,\beta}}$ is well founded. By the properties of $<_{\alpha,\beta}$ its collapse will be $\varphi(\alpha,\beta)$.
\end{proof}

\begin{obs}
    Note that, unlike in the case of second order arithmetic, where well orders being closed under the binary Veblen function is equivalent to Arithmetical Transfinite Recursion, the ordinals being closed under the binary Veblen function does not imply Axiom Beta nor the existence of bisimulations. For example, one may consider $L_{\omega^{\mathrm{CK}}_1}$ in which thef binary Veblen function is total but there exist a tree $T\in L_{\omega^{\mathrm{CK}}_1}$ such that any bisimulation on $T$ is not hyperarithmetic and therefore not in $L_{\omega^{\mathrm{CK}}_1}$.
\end{obs}

\subsection{Stability}

\begin{lemma}
    $\C$ proves that every set has a rank function, that is, for any set $x$, there exists $\alpha\in Ord$ and a surjection $\rho:TC(x)\rightarrow \alpha$ such that
    $$\forall y\in TC(x)\; (\rho(y)=\bigcup\{\rho(z)\cup\{\rho(z)\}: z\in y\}).$$
    We will consider $\rho$ as being a definable function $V\rightarrow Ord$ and write $\rho(x)=\alpha$.
\end{lemma}
\begin{proof}
    Given a set $x$, define $R\subseteq TC(x)\times TC(x)$ to be the transitive closure (in the sense of relations) of the membership relation. The definition of $R$ is $\Delta_0$ relative $(TC(x))^{<\omega}$ and so it exists by $\Delta_0$ Separation. Observe that the collapse of a transitive relation will be a transitive set of transitive sets, that is, an ordinal. The collapsing function of $R$ will be the rank function for $x$.
\end{proof}
\begin{thm}
    $\C$ proves the primitive recursive set functions are total, that is, for  every recursive definition $f$ of a $Prim$ function, $\C$ proves that it defines a total function.  
\end{thm}
\begin{proof}
    We will show something stronger, namely, $\C$ proves for any transitive set $b$, the set $L_{\varphi(\omega,\rho(b)+1)}(b)$\footnote{The $+1$ here is necessary. Consider for example $b=L_{\varphi(\omega+1,0)}$, so $\rho(b)=\varphi(\omega+1,0)$ and therefore  $L_{\varphi(\omega,\rho(b))}(b)=L_{\varphi(\omega,\varphi(\omega+1,0))+\varphi(\omega+1,0)}=L_{2\cdot \varphi(\omega+1,0)}$ which is not closed under $Prim$ functions.} is closed under the primitive recursive set functions. We will do this by induction on the complexity of the recursive definition for $f$, which will be done internal to $\C$, and show that there exists an $i\in \omega$ such that 
    $$\forall \alpha<\varphi(\omega,\rho(b)+1)\, (\rho(b)<\alpha\rightarrow L_{\varphi(i,\alpha)}(b)\vDash f\text{ is total}).$$
    Observe that the induction needed will be $\Delta_0$ relative to $Th(L_{\varphi(\omega,\rho(b)+1)}(b),\in)$. Since each recursive definition of a $Prim$ function is  $\Sigma$, which is upwards absolute, this will imply that $\C$ proves the totality of the primitive recursive set functions.
    \begin{enumerate}
        \item For any of the initial functions we may take $i=0$.
        \item Composition: Let $f(\vec x)=g_0( g_1(\vec x),\dots,g_k(\vec x))$ and $i\in \omega$ large enough such that 
        $$\forall j\leq k\,\forall \alpha<\varphi(\omega,\rho(b)+1)\,(\rho(b)<\alpha\rightarrow L_{\varphi(i,\alpha)}(b) \vDash g_j\text{ is total}\,)$$  
        we will also have
        $$\forall \alpha<\varphi(\omega,\rho(b)+1)\,(\rho(b)<\alpha\rightarrow L_{\varphi(i,\alpha)}(b) \vDash f\text{ is total}\,)$$
        \item Primitive Recursion: Let $f(x,\vec{y})=h(\bigcup\{f(z,\vec{y}):z\in x\},x,\vec{y})$ and $i$ such that
         $$\forall \alpha<\varphi(\omega,\rho(b)+1)\,(\rho(b)<\alpha\rightarrow L_{\varphi(i,\alpha)}(b) \vDash h\text{ is total}\,)$$
        Fix $\alpha>\rho(b)$, note that $\rho(u)+\varphi(i+1,\alpha)=\varphi(i+1,\alpha)$. So it suffices to show that for all $\gamma<\varphi(i+1,\alpha)$ and $\vec{y}\in L_{\gamma}$, by $\Delta_0$ induction on $\beta<\varphi(i+1,\alpha)$ that
        $$L_{\varphi(i,\gamma+\beta)}(b)\vDash \forall x\,(\rho(x)< \beta\rightarrow \exists z\,f(x,\vec{y})=z).$$
        Note that $\beta,\gamma< \varphi(i+1,\alpha)$ implies $\varphi(i,\gamma+\beta)<\varphi(i+1,\alpha)$. 
        The case where $\beta$ is $0$ or limit  is trivial. We prove the successor case. Assume the case for $\beta$ and let $x\in L_{\varphi(i+1, \alpha)}(b)$ be such that $\rho(x)< \beta +1$. Since $\forall w\in x\,(\rho(w)< \beta)$, by inductive hypothesis
        $$L_{\varphi(i,\gamma+\beta)}(b)\vDash \forall w\in x\,\exists z\,( f(w,\vec{y})=z)$$
        and in particular
        $$\bigcup\{ f(w,\vec{y}):w\in x\}\in \mathrm{Def}(L_{\varphi(i,\gamma+\beta)})= L_{\varphi(i,\gamma +\beta)+1}(b)\subseteq L_{\varphi(i,\gamma+\beta+1)}(b).$$
        By the assumption on $h$, we have
       $$L_{\varphi(i,\gamma+\beta+1)}(b)\vDash \exists z\,f(x,\vec{y})= h(\bigcup\{ f(w,\vec{y}):w\in x\},x,\vec{y})=z$$
    \end{enumerate}
    By upwards absoluteness this shows that $L_{\varphi(i+1,\alpha)}(b)\vDash f \text{ is total}$. So for any recursive definition for a primitive recursive set function $f$ we have that $\C\vdash f \text{ is total}$.
\end{proof}

\begin{cor}
    $\C$ proves that for  every transitive set $b$, $L_{\varphi(\omega,\rho(b)+1)}(b)\vDash \mathbf{PRS}\omega$.
\end{cor}
\begin{cor}
    $\mathbf{PRS}\omega+ \text{\upshape Axiom Beta}$ is an extension by definitions of $\C$.
\end{cor}

\section{Primitive recursion and Axiom Beta}

In this section, we  characterize the $\Sigma_1$-definable functions of  the theory 
\[ P\omega\beta \, =\,  \kp\omega^-+ \text{\upshape Axiom Beta}+ \Sigma_1\text{ \upshape Foundation}. \]

Our theory $\C$ is a proper subtheory  of $P\omega\beta$ and its $\Sigma_1$-definable functions do not exhaust those of $P\omega\beta$. In fact, the well founded part $W(a,r)$ of a relation $r$  on a given set $a$  is  $\Sigma_1$-definable in $P\omega\beta$, but not in $\C$.  A complete  characterization of the $\Sigma_1$-definable functions of $\C$ has yet to be established.

For our purposes, let  $\kp$ denote Kripke-Platek set theory without Infinity. Its axioms are Extensionality, Pair, Union, $\Delta_0$ Separation, $\Delta_0$ Collection,  Regularity, and Foundation. Foundation is the scheme
\[ \forall x\, (\forall y\in x\, \vp(y) \to \vp(x)) \to \forall x\, \vp(x). \]
Let $\kp\omega$ stand for $\kp+ \text{ Infinity}$, where Infinity asserts the existence of a limit ordinal (an ordinal is a transitive set of transitive sets). By $\kp^-$, resp.\ $\kp\omega^-$, we mean the  theory obtained by removing Foundation.

\begin{prop}
$\kp\omega^- +\Sigma_1 \text{ \upshape Foundation}  \vdash \text{ \upshape Finite Powerset Axiom}$. 
\end{prop}
\begin{proof}
This is all well known; see Barwise \cite{Barwise}. Let 
\[ \po_{\text{fin}}(a)=\bigcup_{n<\omega} [a]^n, \]
where $[a]^n=\{b\subseteq a: |b|=n\}$. Note that already in $\kp^-+\Sigma_1 \text{ \upshape Foundation}$ one can define the function $n\mapsto [a]^n$ on the natural numbers in a way that conforms to the scheme of (definition by)  $\Sigma$ Recursion  \cite[Theorem I.6.4]{Barwise}. On the other hand, the full strength of Foundation is not needed for the $\Sigma$ Recursion Theorem, and  $\Sigma_1 \text{ \upshape Foundation}$ already suffices. 
\end{proof}
\begin{prop}
$\C\,  \subseteq\,  \kp\omega^-+ \text{\upshape Axiom Beta}+ \Sigma_1\text{ \upshape Foundation}$. 
\end{prop}
\begin{proof}
It is known that already in $\kp^-+\Sigma_1 \text{ \upshape Foundation}$ we can define the transitive closure of a set. The proof for $\kp$ (see \cite[Theorem I.6.1]{Barwise}) goes through in $\kp^-+\Sigma_1 \text{ \upshape Foundation}$.
\end{proof}

\begin{defin}
	A set function $f: V^n\to V$ is $\Xi$-definable if there is a formula $\vp(x_1,\ldots, x_n,y)$ in $\Xi$ such that: 
	\[  f(x_1,\ldots,x_n)= y \text{ iff } \vp(x_1,\ldots,x_n,y).\]
	
	Let $T$ be a theory. We say that $f: V^n\to V$ is $\Xi$-definable in $T$ if there is a formula $\vp(x_1,\ldots, x_n,y)$ in $\Xi$ such that 
	\begin{itemize}
		\item $V\models \vp(x_1,\ldots,x_n,f(x_1,\ldots,x_n))$, for all $x_1,\ldots,x_n\in V$;
		\item  $T\vdash \exists y, \vp(x_1,\ldots,x_n,y)$
		\item $T\vdash  \vp(x_1,\ldots,x_n,y_1) \land  \vp(x_1,\ldots,x_n,y_2)\to y_1=y_2 $.
	\end{itemize}
\end{defin}

Rathjen \cite{R92} gives the following characterization of the primitive recursive set functions.  Hereafter, let 
$g\colon V\to V$ be a $\Delta_0$-definable set function, say by the formula $\vp^g(x,y)$. Let  $$ P^g\, =\, \kp^-+\Sigma_1 \text{ \upshape Foundation} + \forall x\, \exists !y\, \vp^g(x,y).$$

\begin{thm}[Rathjen  \cite{R92}]\label{prim}
  The following are equivalent for any given set function $f:V^n\to V$:
	\begin{enumerate}
		\item $f\in Prim(g)$;
		\item $f$ is $\Sigma_1$-definable in $P^g$;
		\item  $f$ is $\Sigma_1$-definable in $P^g+\Pi_1\text{ \upshape Foundation}$.
	\end{enumerate}
\end{thm}

Since $x \mapsto \omega$ is $\Delta_0$-definable, we obtain the following corollary. 
\begin{thm}[Rathjen  \cite{R92}]\label{prim omega}
	Let $f:V^n\to V$ be a set function. The following are equivalent:
	\begin{enumerate}
		\item $f$ is primitive recursive in $x\mapsto \omega$;
		\item $f$ is $\Sigma_1$-definable in $\kp\omega^-+ \Sigma_1\text{ \upshape Foundation}$.
	\end{enumerate}
\end{thm}

We aim to prove an analogue of \Cref{prim omega} for the theory $P\omega\beta$. It is enough to consider the following function.
\begin{defin}
	Let $B:V\times V\to V$ be defined by letting $B(a,r)=0$ if $r$ is not a well founded relation in $a$, and $B(a,r)=f$ otherwise, where $f$ is the unique collapsing function of $r$ on $a$. 
\end{defin}

Here is the desired characterization. 
\begin{thm}\label{prim in B}
	Let $f:V^n\to V$ be a set function. The following are equivalent:
	\begin{enumerate}
		\item $f\in Prim(x\mapsto\omega, B)$;
		\item $f$ is $\Sigma_1$-definable in $P\omega\beta$.
	\end{enumerate}
\end{thm}

Note that $B$ is $\Sigma_1$-definable in $\kp^-+ \text{\upshape Axiom Beta}$ and that such theory is equivalent to $\kp^-+\forall a,r\, \exists y\, \vp^B(a,r,y)$, where $\vp^B(a,r,y)$ is the $\Sigma_1$ formula defining $B$, namely,

\[   \exists z\, ( (\vartheta(a,r,z) \land y=0) \lor y \text{ is the collapsing function of $r$ on $a$}),  \]  
where $\vartheta(a,r,z)$ is the $\Delta_0$ formula expressing that if $r$ is a relation on $a$ then $z$ is a subset of $a$ with no $r$-minimal element.  \smallskip  

Note that  we cannot obtain our result directly from \Cref{prim}, since $B$ is not $\Delta_0$-definable.   Nonetheless, an adaptation of the Interpretation Theorem (Rathjen  \cite[Theorem 5.2]{R92}), which is the main tool to prove \Cref{prim},  will do the trick. \smallskip

By the end of the section, we will indicate how  \Cref{prim in B} can be further extended by incorporating $g$:

\begin{thm}\label{prim in Bg}
	Let $f:V^n\to V$ be a set function. The following are equivalent:
	\begin{enumerate}
		\item $f\in Prim(x\mapsto\omega, B,g)$;
		\item $f$ is $\Sigma_1$-definable in $P\omega\beta^g$,
	\end{enumerate}
	where 
	
\[ P\omega\beta^g \, =\, P\omega\beta +  \forall x\, \exists !y\, \vp^g(x,y). \] 
\end{thm}

We start with a simple observation. 
\begin{prop}
	Let $T$ be any extension of $\kp^-+\Sigma_1\text{ Foundation}$. If $g$ is $\Sigma_1$-definable in $T$, then every function in $Prim(g)$ is $\Sigma_1$-definable in $T$. 
\end{prop}
\begin{proof}
As already observed, 	definitions by  $\Sigma$ recursion are available in $\kp^-+\Sigma_1\text{ Foundation}$, and this is sufficient to show  that the class of  $\Sigma_1$-definable set functions is closed under primitive recursion.
\end{proof}

Therefore, the only nontrivial implication of \Cref{prim in B} is  $2\Imp 1$.  In the remainder of this section, we  sketch the proof of the Interpretation Theorem for $P^g$, and then explain  how to adapt it to our case. The reader is invited to work out the details along the lines of \cite{R92}.\smallskip


\subsection{Stabilizing functions} 

\begin{defin}
A (relativized) set hierarchy  is a function $H: V\times Ord \to V$ such that:
	\begin{itemize}
		\item $H(b,\alpha)$ is transitive;
		\item $b\subseteq H(b,0)$ and $Ord\subseteq \bigcup_{\alpha} H(b,\alpha)$;
	\end{itemize}
	We write $H_\alpha(b)$ instead of $H(b,\alpha)$.
\end{defin}

\begin{defin}
Let $H: V\times Ord \to V$  be a  set hierarchy and $\vp(x_1,\ldots,x_n)$ be $\Sigma_1$. Say that a function $f: V\times Ord\to V$ is  $H$-stabilizing for $\vp$ if for all $x_1,\ldots,x_n\in H_\alpha(b)$ we have
\[ H_{f(b,\alpha)}(b)\models \vp(x_1,\ldots,x_n). \] 
\end{defin}

The  Interpretation Theorem has the following consequence, which is  interesting by itself and readily implies \Cref{prim}.

\begin{thm}[Rathjen  \cite{R92}]\label{interpretation consequence}
There is   a  set hierarchy  $H$ in $Prim(g)$ such that the following happens.   
If $P^g\vdash \vp(x_1,\ldots,x_n)$,  where $\vp$ is $\Sigma_1$, then there is a set function $f: V\times Ord\to Ord$ in $Prim(g)$ that is $H$-stabilizing for $\vp(x_1,\ldots,x_n)$. 
\end{thm}

For the reader's convenience, we show how to prove  \Cref{prim} from \Cref{interpretation consequence}. Note that the only requirements for $H$ are  that $b\subseteq H_{k(b)}(b)$ for some set function $k\in Prim(g)$ and that every $H_\alpha(b)$ is transitive. 

\begin{thm}[Rathjen  \cite{R92}]
	If $f:V^n\to V$ is $\Sigma_1$-definable in $P^g$, then $f\in Prim(g)$. 
\end{thm}
\begin{proof}
	The proof is standard.
	Let $f$ be $\Sigma_1$-definable via $$\exists z\, \vartheta(x_1,\ldots,x_n,y,z),$$ where $\vartheta$ is $\Delta_0$. By \Cref{interpretation consequence}, there is a set function  $h: V\times Ord\to Ord$ in $Prim(g)$  such that for all $x_1,\ldots, x_n$
	\[  \exists v\in H_\beta(b)\, \exists y\in v\, \exists z\in v\, \vartheta(x_1,\ldots,x_n,y,z), \]
	where $u=\{x_1,\ldots,x_n\}$ and $\beta=h(u,0)$. Let $k(x_1,\ldots,x_n)= H_{h(\{x_1,\ldots,x_n\},0)}(\{x_1,\ldots,x_n\})$. Then $k$ is in $Prim(g)$ and 
	\[  f(x_1,\ldots,x_n)= \bigcup\{ y\in k(x_1,\ldots,x_n)\mid \exists z\in  k(x_1,\ldots,x_n)\, \vartheta(x_1,\ldots,x_n,y,z)\}.  \]
	This shows that $f$ is also in $Prim(g)$. 
\end{proof}

The $H$ in \Cref{interpretation consequence} is played by the following constructible (relativized) set hierarchy, which is based on G\"{o}del operations and does not require the existence of $\omega$: this is the approach taken in Barwise \cite{Barwise}, but with  $g$ added to the fundamental operations. 

\begin{defin}[Constructible hierarchy via fundamental operations]\label{L1}
Let 
\begin{align*}
	L^g_0(b)&= TC(b)\\
	L^g_{\alpha+1}(b)&= D^g(L^g_\alpha(b)\cup \{ L^g_\alpha(b)\})\\
	L^g_\lambda(b) &= \bigcup_{\alpha<\lambda} L^g_\alpha(b), \quad \lambda \text{ limit}
\end{align*}
where
\[  D^g(a)= a\cup TC(\{ f_i(x,y)\mid x,y\in a \land i\le  9\}\cup \{g(x)\mid x\in a\}). \]
Here, $f_0,\ldots,f_9$ are suitable primitive recursive set functions encompassing  the usual G\"{o}del operations ($f_0,\ldots,f_7$)  
\end{defin}

The set function $(b,\alpha)\mapsto L^g_\alpha(b)$ is seen to be primitive recursive in $g$. Note that by construction $g(x)\in L^g_{\alpha+1}(b)$ for every $x\in L^g_\alpha(b)$, and so the function $f(b,\alpha)=\alpha+1$ is stabilizing for $\exists y\, \vp^g(x,y)$. We keep stressing that    $P^g$ can prove the existence of $L^g_\alpha(b)$ for every $\alpha$, despite the lack of Infinity.

\medskip

To prove our  \Cref{prim in B}, we shall consider the usual relativized G\"{o}del's constructible hierarchy obtained by iterating definability.

\begin{defin}[Constructible hierarchy via definability]\label{L2}
Let 
\begin{align*}
	L_0(b)&= TC(b)\\
	L_{\alpha+1}(b)&= Def(L_\alpha(b))\\
	L_\lambda(b) &= \bigcup_{\alpha<\lambda} L_\alpha(b), \quad \lambda \text{ limit}
\end{align*}
where $Def(a)$  is  as usual the set of all  subsets of $a$ definable over $(a,\in)$ with parameters. 
\end{defin}

The set function $(b,\alpha)\mapsto L_\alpha(b)$ is in $Prim(x\mapsto \omega)$.  Now, \Cref{prim in B} follows from the following. 

\begin{thm}
If  $P\omega\beta\vdash \vp(x_1,\ldots,x_n)$,  where $\vp$ is $\Sigma_1$, then there is a set function $f: V\times Ord\to Ord$ in $Prim(x\mapsto\omega,B)$ that is $L$-stabilizing for $\vp(x_1,\ldots,x_n)$.
\end{thm}

The first function to stabilize is $B$ itself. 

\begin{thm}\label{stabilizing B}
There is a function $h: V\times Ord \to Ord$ in $Prim(x\mapsto\omega, B)$ such that, if $a,r\in L_\alpha(b)$, then in  $L_{h(b,\alpha)}(b)$ we can already find the unique collapsing function of $r$ on $a$, if $r$ is well founded on $a$, or   a witness to the fact that  $r$ is not well founded on $a$.
\end{thm}
\begin{proof}
We will prove it for $\alpha\mapsto L_\alpha$, as the same proof works in the relativized case. Suppose $a,r\in L_\alpha$ and $r$ is indeed a relation on $a$. Let $r^+$ denote the transitive closure of $r$: $\pair{x,y}\in r^+$ if there is a finite chain $x_0=x, x_1,\ldots, x_{n+1}=y$ with $\pair{x_i,x_{i+1}} \in r$.  For $x\in a$, we denote by $x+1$ the set $\{x\}\cup\{y\in a \mid \pair{y,x}\in r^+\}$. Note that $x+1\in L_{\alpha+\omega+1}$ since every chain witnessing $\pair{y,x}\in r^+$ can be found in $L_{\alpha+\omega}$. Let $k_0(x,a,r,\alpha)=x+1$ for $x\in a$. This function is in $Prim(x\mapsto\omega)$.  The well founded part of $r$ on $a$ is given by the set $w=\{x\in a \mid r \text{ is well founded on } x+1\}$.  For $x\in w$, let $f_{x+1}: x+1\to V$ be the collapsing function of $r$ on $x+1$. Recall that $\rho$ stands for the rank function on sets.

Now, a routine computation shows that for every $x\in w$ the function $f_{x+1}\in L_{\alpha+\omega\cdot \gamma+ 4}$, where $\gamma=\rho(f_{x+1}(x))$. It easily follows that if $r$ is well founded on $a$, then its collapsing function $f=B(a,r)$ lies in $L_{\alpha+\omega \cdot \rho(f``a)+1}$. Note that $f_{x+1}=B(x+1,r)$ for $x\in w$. Let $k(x,a,r,\alpha)= f_{x+1}(x)$ if $x\in w$ and $0$ otherwise. The function $k$ is in $Prim(k_0,B)$. Define
\[ \gamma(a,r,\alpha)=\bigcup \{ \rho(k(x,a,r,\alpha))+1\mid x\in a \}. \]
Also $\gamma$ is in $Prim(x\mapsto\omega,B)$. By construction,  if $r$ is well founded on $a$, then its collapsing function $f=B(a,r)$ is  in $L_{\alpha+\omega\cdot \gamma(a,r)+1}$, since in this case 
\[  \rho(f``a)=\bigcup\{ \rho(f_{x+1}(x))+1 \mid x\in a \}=\gamma(a,r,\alpha). \] 

On the other hand, if $r$ is not well founded on $a$, $a\setminus w$ is a nonempty subset of $a$ with no $r$-minimal element. We have 
\[ w= \{x\in a\mid \text{ there is a collapsing function of $r$ on $x+1$ in } L_{\alpha+\omega\cdot \gamma(a,r,\alpha)}\}. \]
Thus $a\setminus w$ also lies in $L_{\alpha+\omega\cdot \gamma(a,r,\alpha)+1}$. Define
\[  h(\alpha)=\alpha+\omega\cdot \bigcup\{ \gamma(a,r,\alpha)+1\mid a,r\in L_\alpha\}. \] 
The function $h$ is as desired. 
\end{proof}

We use the set hierarchy based on definability because the set function stabilizing Axiom Beta seems to require $\omega$, but   $x\mapsto \omega$ is not primitive recursive in  $B$.

\subsection{Interpretation theorem}

We first state the interpretation theorem for $P^g$  and finally show  how to extend it to $P\omega\beta$. The proof uses a Tait-style sequent calculus. Sequents are finite sets of formulas in negation normal form, denoted $\Gamma,\Delta,\ldots$. The intended meaning of $\Gamma$ is the disjunction of all formulas in $\Gamma$. By $\Gamma,\vp$ we mean $\Gamma\cup\{\vp\}$. We denote variables by $a,b,\ldots,x,y,z$. \smallskip

\begin{defin}[The calculus $LP$]
Axioms.
\begin{itemize}
	\item Logical axioms: $\Gamma,\vp,\neg\vp$, for  $\vp\in \Delta_0$
	\item Equality axioms: $\Gamma, a=b\land \vp(a)\to \vp(b)$, for $\vp\in\Delta_0$
	\item Extensionality: $\Gamma, \forall x\in a\, (x\in b)\land \forall x\in b\, (x\in a)\to a=b$
	\item Pair: $\Gamma, \exists y\, (y=\{a,b\})$
	\item Union: $\Gamma, \exists y\, (y=\bigcup a)$
	\item $\Delta_0$ Separation: $\Gamma, \exists y\, (y=\{x\in a\mid \vp(x)\}$, for  $\vp\in \Delta_0$
	\item Regularity: $\Gamma, a\neq\emptyset\to \exists x\in a\, (x\cap a=\emptyset)$ 
\end{itemize}

Normal logical rules.

\[  \begin{array}{lcr} 
	\ \prf{\Gamma,\vp \quad \Gamma,\psi \line \Gamma,\vp\land\psi} \quad \land && 
	\prf{\Gamma,\vp_i \line \Gamma,\vp_1\lor\vp_2} \quad \lor  \\[7mm]
	\prf{\Gamma, a\in b\to \vp(a)\line \Gamma,\forall x\in b\, \vp(x)} \quad b\forall &&  \prf{\Gamma, a\in b\land \vp(a)\line \Gamma,\exists x\in b\, \vp(x)}\quad b\exists  \\[7mm]
	\prf{\Gamma, \vp(a)\line \Gamma,\forall x\, \vp(x)} \quad \forall &&  \prf{\Gamma, \vp(a)\line \Gamma, \exists x\, \vp(x)}\quad \exists 
\end{array}  \] 

\medskip 
As usual, the variable $a$ is free for  $x$ in $\vp(x)$. In $b\forall$ and  $\forall$, the variable $a$ must further  satisfy the eigenvariable condition:   $a$ is not to occur free in the lower sequent.
\medskip

Cut rule.

\[ \prf{\Gamma, \vp\quad \Gamma,\neg\vp\line \Gamma} \quad Cut\] 

The formulas $\vp,\neg\vp$ are the cut formulas. 
\medskip

Nonlogical rules.  We have a rule to derive $\Delta_0$ Collection. 

\[ \prf{ \Gamma, \forall x\in a\, \exists y\, \vp(x,y)\line \Gamma, \exists z\, \forall x\in a\, \exists y\in z\, \vp(x,y) }\] 
\medskip 
where $\vp$ is $\Delta_0$. Finally, there is a rule to derive $\Sigma_1$ Foundation. 

\[ \prf{\Gamma, \neg \forall x\in a\, \exists y\, \vp(x,y), \exists y\, \vp(a,y)\line \Gamma, \exists y\, \vp(b,y)} \]
\medskip 
where $\vp(x,y)$ is $\Delta_0$, the variables $a$ and $b$ are free for $x$ in $\exists y\, \vp(x,y)$, and $a$ is not to occur free in the sequent $\Gamma, \forall x\, \exists y\, \vp(x,y)$. 
\end{defin}

\begin{defin}[The calculus $LP^g$]
We add the following axioms to $LP$.
\begin{itemize}
\item $\Gamma, \exists y\, \vp^g(a,y)$
\item $\Gamma, \vp^g(a,b)\land \vp^g(a,c)\to b=c$
\end{itemize} 
\end{defin}

\begin{defin}
Write $ LP^{g}\vdash \Gamma$ if there is a derivation of the sequent $\Gamma$ in $LP^{g}$. 
\end{defin}

In this calculus, we can prove every theorem of $P^g$.

\begin{lemma}[Embedding]\label{embedding}
	If $P^g\vdash \vp$, then $LP^g \vdash \vp$. 
\end{lemma} 

Note that the principal formula of every axiom and of  every nonlogical rule is $\Sigma_1$. This allows us to eliminate all cuts but the ones with cut formulas in $\Sigma_1\cup\Pi_1$. 

\begin{lemma}[Partial cut elimination]
	If $ LP^g \vdash \Gamma$, then there is a derivation 	where all cut formulas are in $\Sigma_1\cup\Pi_1$.
\end{lemma}

\begin{defin}
	The set of $\Delta_0(\Sigma_1\cup\Pi_1)$ formulas is the smallest containing all $\Delta_0\cup \Sigma_1\cup \Pi_1$ formulas and closed under $\land,\lor$ and bounded quantifiers. 
\end{defin}
Since the minor formula of every nonlogical axiom is $\Delta_0(\Sigma_1\cup\Pi_1)$,  one readily verifies the following. 
\begin{lemma}
	Suppose $\Gamma$ consists of $\Delta_0(\Sigma_1\cup\Pi_1)$ formulas. If $LP^g \vdash \Gamma$, then $ LP^g \vdash\Gamma$ with a derivation consisting only of $\Delta_0(\Sigma_1\cup\Pi_1)$.
\end{lemma}
\begin{proof}
	By induction on a derivation where all cuts are in $\Sigma_1\cup\Pi_1$. 
\end{proof}

\begin{thm}[{Interpretation Theorem for $P^g$;  \cite[Theorem 5.2]{R92}}]\label{interpretation Tg}
	Let $\Gamma(x_1,\ldots,x_n)$ consist of $\Delta_0(\Sigma_1\cup\Pi_1)$ formulas with free variables in $x_1,\ldots,x_n$. If $$ LP^g \vdash \Gamma(x_1,\ldots,x_n),$$ then there is a function $f: V\times Ord\to Ord$ in $Prim(g)$ such that for all  $\alpha, b$ and for all sets $x_1,\ldots,x_n\in L^g_\alpha(b)$, we have that
	\[   \Gamma(x_1,\ldots,x_n)^{\alpha,\beta} \] 
	holds true with $\beta=f(b,\alpha)$. Here, 
	\[ \vp^{\alpha,\beta} \] 
	denotes the formula obtained from $\vp$ by restricting the unbounded universal quantifiers to $L^g_\alpha(b)$ and the unbounded existential quantifiers to $L^g_\beta(b)$. Finally,
	\[ (\vp_1,\ldots, \vp_n)^{\alpha,\beta} \] 
	stands for the disjunction
	\[ (\vp_1)^{\alpha,\beta}\lor \ldots\lor (\vp_n)^{\alpha,\beta}. \]
\end{thm}

\Cref{interpretation consequence} follows immediately from \Cref{embedding} and  \Cref{interpretation Tg}.

We can finally show how to adapt the Interpretation Theorem.  A sequent calculus $LP\omega\beta$ for $P\omega\beta$  is obtained from $LP$ by adding the following $\Sigma_1$ axioms.
\begin{itemize}
	\item Infinity: $\Gamma, \exists x\, Lim(x)$
	\item Axiom Beta: $\Gamma, \exists y\, \vp^B(a,r,y)$
\end{itemize}

\begin{thm}[Interpretation Theorem for $P\omega\beta$]
Let $\Gamma(x_1,\ldots,x_n)$ consist of $\Delta_0(\Sigma_1\cup\Pi_1)$ formulas with free variables in $x_1,\ldots,x_n$. If $$ LP\omega\beta \vdash \Gamma(x_1,\ldots,x_n),$$ then there is a function $f: V\times Ord\to Ord$ in $Prim(x\mapsto\omega,B)$ such that for all  $\alpha, b$ and for all sets $x_1,\ldots,x_n\in L_\alpha(b)$, we have that
\[   \Gamma(x_1,\ldots,x_n)^{\alpha,\beta} \] 
holds true with $\beta=f(b,\alpha)$. Here, 
\[ \vp^{\alpha,\beta} \] 
denotes the formula obtained from $\vp$ by restricting the unbounded universal quantifiers to $L_\alpha(b)$ and the unbounded existential quantifiers to $L_\beta(b)$.	
\end{thm}
\begin{proof}
Except for the choice of the  set hierarchy, the proof follows the one in Rathjen  \cite[Theorem 5.2]{R92}, by using \Cref{stabilizing B} to interpret Axiom Beta. 
\end{proof}

One can generalize to $P\omega\beta^g$, and thereby obtain \Cref{prim in Bg},   by  considering the following relativized set hierarchy.

\begin{defin}
	Let
\begin{align*}
	L^g_0(b)&= TC(b)\\
	L^g_{\alpha+1}(b)&= Def^g(L^g_\alpha(b))\\
	L^g_\lambda(b) &= \bigcup_{\alpha<\lambda} L^g_\alpha(b), \quad \lambda \text{ limit}
\end{align*}
where 
\[  Def^g(a)=Def(a\cup \{a\}\cup TC(\{g(x)\mid x\in a\})). \]
\end{defin}
The set function $(b,\alpha)\mapsto L^g_\alpha(b)$ is now in $Prim(x\mapsto\omega,g)$.

\section{On the Necessity of the Finite Powerset Axiom}
In this section, we show that the Finite Powerset Axiom cannot be replaced by the Cartesian Product Axiom, nor an axiom that states the universe is closed under rudimentary functions. We also show that with the Axiom of  Countability, the formulation of $\atrset$ given in \cite{simpson82} is equivalent to the one given in \cite{Simpson} as well as our system $\C+\text{ Countability}$. \smallskip

The first formulation of $\atrset$ found in \cite[Section 2]{simpson82} is synonymous with the system $$\mathbf{PRS}\omega+\,\text{Axiom Beta},$$
which is an extension by definitions of $\C$. The Axiom of Countability is taken to be optional. The fact that $\mathbf{PRS}\omega+\text{Axiom Beta}$ is an extension by definitions of $\atrset$, in this formulation without Countability, was shown by Freund \cite[Corollary 1.4.6]{Freundthesis}.  \\
\\
The second formulation of $\atrset$ found in \cite[Definitions VII.3.3 \& VII.3.8]{Simpson} has as axioms:
\begin{enumerate}
    \item Extensionality;
    \item Infinity (Simpson's formulation is the following $\exists x\, (\emptyset\in x \wedge \forall u,v\in x(u\cup \{v\}\in x))$, which is equivalent, relative to the other axioms, to the statement $V_\omega$ exists.);
    \item The universe is closed under the basic rudimentary functions $F_0,\dots, F_8$ (see \Cref{rud}, Simpson also includes the inverse map $x\mapsto x^{-1}=\{\pair{u,v}: \pair{v,u}\in x\}$ which is redundant as the closure under $F_0,\dots,F_8$ implies the closure under the rudimentary functions. See \cite[Property 1.3 \& Lemma 1.8]{jensen}, \cite[Theorem 1.4.5]{gandy74}, and \cite[Remark 1.12]{mathias06});
    \item Regularity;
    \item Transitive containment\footnote{Simpson does not include this axiom as it is a direct consequence of the Axiom of  Countability.};
    \item Axiom Beta;
    \item Countability (every set is contained in a countable transitive set).
\end{enumerate}
We will show in this section that the two formulations with the Axiom of  Countability are equivalent while without the Axiom of  Countability, the first formulation is strictly stronger than the second formulation. For this purpose, we introduce the system $\C_\times$ which is obtained by replacing the Finite Powerset Axiom of $\C$ with  the Cartesian Product Axiom. We first show that $\C_\times$ is equivalent to the second formulation without Countability.
\noindent 
\begin{lemma}
$\C_\times$ proves the existence of the graphs of addition and multiplication on $\omega$, and that $V_\omega$ exists.
\end{lemma}
\begin{proof}
We show that the existence of the graph of addition $+:\omega\times\omega\to \omega$. We define a relation $R\subseteq \omega^3\times \omega^3$ such that if $\pi$ is the collapsing function for $R$, then $\pi\pair{0,0,0}$ will be the graph of addition. We define $R$ by the following rules:
		\begin{itemize}[leftmargin=7mm]
			\item $\pair{1,n,m}\, R\, \pair{0,0,0}$. (we will have $\pi\pair{1,n,m}=\pair{\pair{n,m},n+m}$);
			\item $\pair{2,n,m}, \pair{3,n,m} \, R\, \pair{1,n,m}$;
			\item $\pair{4,n,m}\, R\, \pair{2,n,m}$;
			\item $\pair{4,n,m}, \pair{7,n,m}\, R\, \pair{3,n,m}$;
			\item $\pair{5,n,0}, \pair{6,n,m}\, R\, \pair{4,n,m}$;
			\item $\pair{7,n,0}\, R\, \pair{5,n,0}$;
			\item $\pair{7,n,0}, \pair{7,m,0}\, R\, \pair{6,n,m}$;
			\item $\pair{7,i,0}, \pair{7,n,j}\, R\, \pair{7,n,m}$ whenever $i\in n$ and $j\in m$.
		\end{itemize}
 One can verify that the relation $R$ is well founded and the collapse $\pi\pair{0,0,0}$ defines the graph of a function $\omega\times \omega\rightarrow \omega$ which satisfies the inductive rules of addition.  Similarly, one defines multiplication. Relative to $\omega$ and the graph of addition and multiplication, any arithmetic predicate will be $\Delta_0$, so in particular, $\C_\times$ proves that the reals satisfy $\mathbf{ACA}_0$. So $\C_\times$ can prove the existence of the Ackermann relation on $\omega$ and proves that its collapse is $V_\omega$.
\end{proof}
\begin{lemma}
    $\C_\times$ proves the functions $F_0,\dots, F_8$ are total.
\end{lemma}
\begin{proof}
    Since Pair, Union, Cartesian Product, and $\Delta_0$ Separation are axioms of $\C_\times$ it suffices to show that $F_8(x,y)=\{\{w:\pair{w,z}\in x\}:z\in y\}$ is total. Given sets $x$ and $y$ define on the set $$(TC(x)\times \{0\})\cup (y\times \{1\})\cup\{\pair{0,2}\}$$  the relation $R$ given by the following three rules
    \begin{enumerate}
        \item $\pair{v,1}\,R \,\pair{0,2}\;\leftrightarrow \;v\in y$
        \item $\pair{u,0}\,R\,\pair{v,1}\;\leftrightarrow\; v\in y\wedge \pair{u,v}\in x$.
        \item $\pair{u,0}\,R\,\pair{v,0}\;\leftrightarrow \;u\in v$.
    \end{enumerate}
    Note that $\bigcup \bigcup x\subseteq TC(x)$. Let $\pi$ be the collapse of $R$, $\pi\pair{0,2}=\{\{w:\pair{w,z}\in x\}:z\in y\}$. 
\end{proof}

    This shows that  $\C_\times+\text{ Countability}$  is precisely the same as the second formulation of $\atrset$. Simpson showed that $\atrset$ is the unique theory which is interpreted by $\atr$ via the tree interpretation \cite[Theorem VII.3.22]{Simpson}. By this result, the two formulations with Countability will coincide. 

We finally show that in the absence of Countability $\C_\times$ is a strictly weaker system than $\C$.  To do this, we first need to introduce the first Cohen model. 
\begin{defin}
    The first Cohen model, or the Cohen-Halpern-Levy model, is the symmetric extension obtained by the system $(\text{Add}(\omega,\omega),S_\omega,\mathcal{F})$ where 
    \begin{enumerate}
        \item $\text{Add}(\omega,\omega)$ is the collection of partial functions with finite domain $p:\omega\times \omega\rightarrow 2$ ordered under reverse inclusion.
        \item $S_\omega$ is the group of all bijections $\omega\rightarrow \omega$ that act on $\text{Add}(\omega,\omega)$ by
        $$(\pi, p)\mapsto( (n,m)\mapsto p(\pi(n),m)).$$
        \item $\mathcal{F}$ is the filter of subgroups generated by the collection
        $$\text{fix}(E)=\{\{\pi\in S_\omega: \pi\vert_E=id_E\}:E\in \mathcal{P}_\text{fin}(\omega)\}.$$
    \end{enumerate}
    The action of $S_\omega$ extends naturally to all $\text{Add}(\omega,\omega)$-names. Given an $\text{Add}(\omega,\omega)$-name $\dot x$ its subgroup of symmetries is the set
    $$\text{sym}(\dot x)=\{\pi \in \text{Add}(\omega,\omega):\pi\dot x=\dot x\}.$$
    We say that $\dot x$ is symmetric if $\text{sym}(\dot x)\in \mathcal{F}$. We define recursively the hereditary symmetric names as being all symmetric names $\dot x$ such that for all $\dot y$ and $p\in \text{Add}(\omega,\omega)$ if $\pair{p,\dot y}\in \dot x$ then $\dot y$ is hereditary symmetric. We say that a finite set $E\subseteq \omega$ is a support for a hereditary symmetric name $\dot x$ if $\text{fix}(E)\subseteq \text{sym}(\dot x)$ and it is a support for a condition $p$ if  $$\forall \pair{n,m}\in \text{dom}(p)\,( n\in E).$$
    Fix a countable transitive model $M\vDash \zfc$, let $G$ be an $M$-generic filter for $\text{Add}(\omega,\omega)$ and $N$ the respective symmetric extension, that is, the interpretation of the hereditary symmetric names by $G$ (see \cite[5.1.2]{jechac}). A standard forcing argument shows that $\bigcup G:\omega\times \omega\rightarrow 2$ is a total function. For each $n\in \omega$ let $\dot{a}_n=\{\pair{p,\check{m}}:p(n,m)=1\}$ and $\dot{A}=\{\pair{\emptyset,\dot{a}_n}:n\in \omega\}$ be the canonical names for $a_n=\{m\in \omega:\bigcup G(n,m)=1\}$ and $A=\{a_n:n\in \omega\}$ respectively. We will denote the generic extension by $G$ as $M[G]$; note that $M\subseteq N\subseteq M[G]$. One of the main properties of this model is that the set $A$ is a Dedekind finite ($A$ is not equipotent to any of its proper subsets) infinite set. For a more detailed exposition on the first Cohen model, see \cite[Section 5.3]{jechac} or \cite[pages 398-400]{Habelsein}.
\end{defin}

\begin{defin}
    For any pair of sets $A$ and $B$ we write $|A|\leq^*|B|$ to mean there exists a partial surjection $f:B\rightarrow A$, $|A|\nleq^*|B|$ if no such surjection exists and $|A|<^*|B|$ if $|A|\leq^*|B|\wedge |B|\nleq^*|A|$. A set $B$ is said to be dual Dedekind finite if $|B|+1\nleq^*|B|$.
\end{defin}
Truss \cite[Section 5]{Truss} showed that in the first Cohen model, the set $A$ is dual Dedekind finite. With a bit more work, we get the following generalization.
\begin{lemma}
    In the first Cohen model, for every $n\in \omega$, the set $A^{n}$ is dually Dedekind finite.
\end{lemma}
\begin{proof}
       We point out that, unlike the Dedekind finite cardinals, the dual Dedekind finite cardinals may not be closed under product as shown by Truss \cite[Section 3]{Truss}. The case $n=0$ is trivial, so we may assume $n>0$ and therefore $\emptyset\notin A^n$. Seeking a contradiction, assume that there exists a surjection $f:A^n\rightarrow A^n\cup\{\emptyset\}$ in $N$. Since $f\in N$, it will have a symmetric name $\dot{f}$. Since $f$ being a surjection from $A$ to $A\cup\{\emptyset\}$ can be expressed by a $\Delta_0$ formula with parameters $f$ and $A\cup\{\emptyset\}$, by absoluteness, $f$ will also be a surjection in $M[G]$. By the forcing theorem (see \cite[VI.2.44]{kunen}) there exists $p\in \text{Add}(\omega,\omega)$ such that 
       $$p\Vdash \dot{f}:\dot{A}^n\rightarrow \dot{A}^n\cup\{\emptyset\} \text{ is a surjective function}\,.$$
       Fix a finite set $E\subseteq \omega$ which is a support for $p$ and $\dot{f}$. The function $f\in M[G]$ defines a unique function $h:\omega^n\rightarrow \omega^n\cup\{\emptyset\}$, which is also in $M[G]$, where 
       $$f(a_{\sigma(0)},\dots a_{\sigma(n-1)})=\begin{cases}
       \emptyset \quad\quad\quad\quad \text{ if and only if\quad } h(\sigma)=\emptyset\\
           (a_{h(\sigma)(0)},\dots a_{h(\sigma)(n-1)})\quad \;\;\, \text{ otherwise}
       \end{cases}$$
       Since $f$ is surjective, $h$ will also be surjective. We will treat elements of $\omega^n$ as sequences of length $n$.
       Seeking a contradiction, assume that for every $\tau\in \omega^n$ we have
       $$\text{rng}(h(\tau))\subseteq \text{rng}(\tau)\cup E\vee h(\tau)=\emptyset.$$
       Let $R_j=\{ \tau\in \omega^n:|\text{rng}(\tau)\setminus E|\leq j\}$, by hypothesis we must have that $h(R_j)\subseteq R_j\cup\{\emptyset\}$.
       We prove by induction on $j\leq n$ that $h\vert_{R_j}:R_j\rightarrow R_j\cup\{\emptyset\}$ is not surjective, which would contradict our initial assumption. For the base case, note that $R_0= E^n$ is a finite set and therefore cannot surject onto $E^n\cup\{\emptyset\}$. Assume that $h\vert_{R_j}:R_j\rightarrow R_j\cup \{\emptyset\}$ is not a surjection for some $j<n$. If $h\vert_{R_{j+1}}:R_{j+1}\rightarrow R_{j+1}\cup\{\emptyset\}$ were a surjection then some element $\tau\in R_{j+1}\setminus R_{j}$ must be set to $R_j\cup\{\emptyset\}$. The set $(E\cup\text{rng}(\tau))^n\setminus R_j$ is finite. Let $\sigma\in R_{j+1}$ be such that $h(\sigma)\in (E\cup \text{rng}(\tau))^n\setminus R_j$, by assumption  
       $$\text{rng}(h(\sigma))\subseteq \text{rng}(\sigma)\cup E \quad \text{ and }\quad \sigma\in R_{j+1}\setminus R_j.$$  Since $\tau,\sigma,h(\sigma)\in R_{j+1}\setminus R_j$ we have that $$|\text{rng}(h(\sigma))\setminus E|=|\text{rng}(\sigma)\setminus E|=|\text{rng}(\tau)\setminus E|=j+1$$
       and since $h(\sigma)\in (\text{rng}(\sigma)\cup E)^n$ we have $\text{rng}(h(\sigma))\subseteq \text{rng}(\tau)\cup E$. So
       $$\text{rng}(\tau)\setminus E=\text{rng}(h(\sigma))\setminus E=\text{rng}(\sigma)\setminus E.$$
       Therefore, $\text{rng}(\sigma)\subseteq \text{rng}(\tau)\cup E$. This proves that
       $$h^{-1}((E\cup\text{rng}(\tau))^n\setminus R_j)\subseteq (E\cup \text{rng}(\tau))^n\setminus R_j$$
       Since $(E\cup \text{rng}(\tau))^n\setminus R_j$ is finite and $h(\tau)\notin (E\cup \text{rng}(\tau))^n\setminus R_j$ this implies $h\vert_{R_{j+1}}$ does not surject onto $R_{j+1}\cup\{\emptyset\}$. So $h:\omega^n\rightarrow \omega^n\cup\{\emptyset\}$ and $f:A^n\rightarrow A^n\cup\{\emptyset\}$ are not surjective which contradicts our initial assumption. 

       So assume there exists a $\tau\in \omega^n$ such that $\text{rng}(h(\tau))\not\subseteq \text{rng}(\tau)\cup E$ and $h(\tau)\neq \emptyset$. Fix an $i\in\omega $ such that $h(\tau)(i)\notin (\text{rng}(\tau)\cup E)$ and $q\leq p$ such that
       $$q\Vdash \dot{f}(\dot a_{\tau(0)},\dots, \dot a_{\tau(n-1)})=(\dot a_{h(\tau)(0)},\dots, \dot a_{h(\tau)(n-1)}).$$
       Fix a $j\in \omega$ which is not in the support of $q$ and in $E\cup \text{rng}(\tau)$ and let $\pi$ be the transposition which sends $j$ to $h(\tau)(i)$. By the choice of $j$, we have that $q$ and $\pi q$ are compatible. Furthermore we have
       $$\pi q\Vdash \dot{f}(\dot a_{\tau(0)},\dots, \dot a_{\tau(n-1)})=(\dot a_{h(\tau)(0)},\dots,\dot{a}_{h(\tau)(i-1)}, \dot{a}_{j},\dot{a}_{h(\tau)(i+1)}, \dot a_{h(\tau)(n-1)}).$$
       This implies, however $q\cup \pi q$ forces that $f(a_{\tau(0)},\dots, a_{\tau(n-1)})$ assumes two distinct values, contradicting the initial assumption that $$p\Vdash \dot f \text{ is a function}.$$
       This proves that $A^n$ is dual Dedekind finite.
\end{proof}

\begin{obs}
    In general, $\mathbf{ZF}$ proves that any infinite subset of $\mathcal{P}(\omega)$ surjects onto $\omega$. To see this, first note that for an infinite set $A\subseteq 2^\omega$ the map $\sigma\mapsto A\cap[\sigma]$ must have infinite range. This induces an injection $\omega\rightarrow \mathcal{P}(A)$ which, by a theorem of Kuratowski\footnote{Tarski \cite[Pages 94-95]{Tarski} attributes this result to Kuratowski.} (See \cite[Proposition 5.3]{Habelsein} or \cite[Lemma 4.11]{Horst}), implies that $\aleph_0\leq^*|A|$ and in our case $\aleph_0<^*|A|$. 
\end{obs}

\begin{cor}
    In the first Cohen models the cardinality $\mathfrak{a}$, of the set $A\subseteq \mathcal{P}(\omega)$,  satisfies
$$\aleph_0<^*\mathfrak{a}<^*\mathfrak{a}^2<^*\mathfrak{a}^3<^*\dots <^* \mathfrak{a}^{<\omega}=|A^{<\omega}|$$
\end{cor} 
\noindent
 The existence of such a set of reals $A$ was proved consistent with $\mathbf{ZF}$ by Krivine \cite{krivine}, using realizability algebras, and by Karagila \cite{Karagila} using symmetric extensions. Their construction has the advantage of also satisfying $\mathbf{DC}$.

\begin{thm}\label{amphmod}
 $\C_\times+\text{Separation}\not\vdash \text{ Finite Powerset}$.

\end{thm}
\begin{proof}
    Working in a model of $\mathbf{ZF}$ where there exists a set $A\subseteq \mathcal{P}(\omega)$ such that for every $n\in \omega$ $\aleph_0<^*|A^n|<^*|A^{n+1}|<^*|A^{<\omega}|$ define
    $$M=\{x:\exists n\in \omega\, |TC(x)|\leq^* |A^n|\}.$$
    We show that $M$ is the desired model. Since $A$ is an infinite set of reals, $TC(A)=\omega \cup A$. Since $\aleph_0<^*|A|$ we have $|TC(A)|\leq^* 2|A|\leq^* |A^2|$ so $A\in M$ but $A^{<\omega}\notin M$. $M$ is a transitive class which is closed under subset, that is $y\subseteq x\in M\rightarrow y\in M$, so it satisfies Extensionality, Regularity, Separation, and by construction it also satisfies Transitive containment and Union. Since $A$ surjects on $\omega$, $M$ satisfies the Axiom of  Infinity. To see that $M$ satisfies Axiom Beta, first note that, since $M$ is closed under subsets, a relation $R$ is well founded if and only if it is well founded in $M$. Given a set with a well founded relation $\pair{X,R}\in M$, by definition of $M$, there must be an $A^n$ which surjects onto $X$ and will therefore also surject onto the collapse of $R$ and on its collapsing function. For Cartesian Product, note that $$TC(x\times y)= TC(x)\cup TC(y)\cup\{\{u,v\}:u\in x\wedge v\in y\}\cup \{\{v\}:v\in y\}\cup x\times y$$ so given an $n\in \omega$ such that $|TC(x)|\leq^*|A^n|\wedge |TC(y)|\leq^*|A^n|$, one has $|TC(x\times y)|\leq^*|A^{2n+1}|$. The Axiom of  Pairing is proved similarly.
\end{proof}

The last thing to note is that while $\C_\times$ cannot prove the existence of $L_\alpha(b)$ for all $a\in V$ and $\alpha\in Ord$, we do have that the finite powerset on ordinals is a rudimentary function. So the construction of the unrelativized constructible hierarchy should still be possible.
\begin{question}
    What is the strength of the Mostowski collapse lemma, that is, Axiom Beta restricted to extensional relations? In the presence of the Axiom of Countability it still ensures that the reals satisfy $\mathbf{ATR}_0$ by the comparability of well orders. In absence of Countability it is not clear whether it still implies $\Delta_0\text{-}\mathbf{TR}$.
\end{question}
\begin{question}
    The existence of maximal bisimulations on all trees, not just well founded, can be shown to be equivalent to $\Pi^1_1\text{-}\mathbf{CA}_0$, which is also equivalent to the existence of maximal fixed points for arithmetic operators. A natural question is whether the existence of maximal fixed points for $\Delta_0$ operators is equivalent to the existence of maximal bisimulations. We conjecture that this is possible. The importance of maximal fixed points is also related to the interpretation of what Aczel terms the Anti-Foundation Axiom  ($\mathbf{AFA}$).\footnote{The generalization of Axiom Beta to arbitrary relations appears to have been first considered by Scott and has become a central tool in computer science following the work of Forti and Honsell \cite{Fonsell}; the principle was  subsequently christened $\mathbf{AFA}$ (Anti-Foundation Axiom) by Aczel \cite{Aczel} and further investigated in the context of  set theory with non-well-founded sets.} Let $\mathbf{T}$ denote the theory $\B$ with the Axiom of Regularity replaced with Aczel's $\mathbf{AFA}$. It should be possible to interpret $\mathbf{T}$ in the system $\Pi^1_1\text{-}\mathbf{CA}_0$. On the other hand, modifying a proof of Rathjen \cite[Theorem 3.9]{Rathjen01}, one can show that $\mathbf{T}$ proves the reals satisfy $\Pi^1_1\text{-}\mathbf{CA}_0$.
\end{question}

\section*{Acknowledgements}
The  first author was  partially supported by the project PRIN 2022 “Models, Sets and Classifications” Code no.\ 2022TECZJA CUP G53D23001890006, funded by the European Union – NextGenerationEU – PNRR M4 C2 I1.1.  
The second author would like to thank Asaf Karagila, Alvaro Pintado, Paul Shafer, and Giovanni Soldà for helpful discussions. 
The authors would like to acknowledge the support of the Erwin Schr\"{o}dinger Institute in Vienna
as part of the thematic program “Reverse Mathematics” in 2025.
\bibliographystyle{plain}
\bibliography{biblio}

\end{document}